\documentclass[10pt,reqno]{amsart}
\usepackage[svgnames,table]{xcolor} 
\usepackage[colorlinks=true,citecolor=Magenta,linkcolor=Cyan]{hyperref} 
\usepackage{amsmath,amsfonts,latexsym,amssymb}
\usepackage{amssymb,amsfonts, amsmath}
\usepackage{mathrsfs}
\usepackage[latin1]{inputenc}
\usepackage[T1]{fontenc}
\usepackage{ae,aecompl}
\usepackage{braket}
\usepackage{comment}
\usepackage{amssymb,amsfonts, amsmath}
 \usepackage{mathtools} 
\usepackage[all,arc]{xy}
\usepackage{enumerate}
\usepackage{mathrsfs}
\usepackage{mathabx}
\usepackage{amssymb}
\usepackage{dsfont}
\usepackage{graphicx}
\usepackage{subfig}
\usepackage[nameinlink]{cleveref}
\usepackage{verbatim} 
\usepackage{amssymb,amsfonts, amsmath}
\usepackage[all,arc]{xy}
\usepackage{enumerate}
\usepackage{mathrsfs}
\usepackage{hyperref}
\usepackage{xstring}
\usepackage{amsmath}
\usepackage[makeroom]{cancel}
\newtheorem{theorem}{Theorem}[section]
\newtheorem{lemma}[theorem]{Lemma}
\newtheorem{proposition}[theorem]{Proposition}
\newtheorem{corollary}[theorem]{Corollary}
\newtheorem{definition}[theorem]{Definition}

\newtheorem{observation}[theorem]{Observation}

\newtheorem{rmks}[theorem]{Remarks}
\renewcommand{\leq}{\leqslant}
\renewcommand{\geq}{\geqslant}

\usepackage{braket}
\usepackage{amsmath}
\usepackage[title]{appendix}
\makeatletter
\renewcommand*\env@matrix[1][\arraystretch]{%
\edef\arraystretch{#1}%
\hskip -\arraycolsep
\let\@ifnextchar\new@ifnextchar
\array{*\c@MaxMatrixCols c}}
\makeatother
\date{\today}
\AtEndDocument{\bigskip{\footnotesize%
\textsc{Max Planck Institute for Mathematics in the Sciences, Inselstra{\ss}e 22, 04103 Leipzig, Germany } \par  
 \textit{E-mail address}: \texttt{konstantinos.tsouvalas@mis.mpg.de} }}
\title{The H\"older exponent of Anosov limit maps}
\date{\today}

\author{Konstantinos Tsouvalas}
\begin{document}

\frenchspacing

\maketitle

\begin{abstract} Let $\Gamma$ be a non-elementary hyperbolic group and $d_{a}, a>1,$ a visual metric on its Gromov boundary $\partial_{\infty}\Gamma$. For a $1$-Anosov representation $\rho:\Gamma \rightarrow \mathsf{GL}_{d}(\mathbb{K})$, where $\mathbb{K}=\mathbb{R}$ or $\mathbb{C}$, we calculate the H\"older exponent of the Anosov limit map $\xi_{\rho}^1:\big(\partial_{\infty}\Gamma, d_{a}\big)\rightarrow \big(\mathbb{P}(\mathbb{K}^d),d_{\mathbb{P}}\big)$ of $\rho$ in terms of the moduli of eigenvalues of elements in $\rho(\Gamma)$ and the stable translation length on $\Gamma$. If $\xi_{\rho}^1(\partial_{\infty}\Gamma)$ spans $\mathbb{K}^d$ and $\rho$ is either irreducible or $\{1,2\}$-Anosov, then we prove that $\xi_{\rho}^1$ attains its H\"older exponent. We also provide an analogous calculation for the exponent of the inverse of the Anosov limit map of $(1,1,2)$-hyperconvex representations, including the class of Hitchin representations. Finally, we exhibit examples of (non semisimple) $1$-Anosov representations of surface groups in $\mathsf{SL}_4(\mathbb{R})$ whose Anosov limit map in $\mathbb{P}(\mathbb{R}^4)$ does not attain its H\"older exponent.\end{abstract}

\section{Introduction}
Anosov representations form a rich and stable class of discrete subgroups of linear reductive Lie groups with special dynamical properties, introduced by Labourie \cite{Labourie} in his work on the Hitchin component. Labourie's dynamical definition was further generalized by Guichard--Wienhard \cite{GW} for arbitrary word hyperbolic groups. Anosov representations have been extensively studied by Kapovich--Leeb--Porti \cite{KLP2}, Gu\'eritaud--Guichard--Kassel--Wienhard \cite{GGKW}, Bochi--Potrie--Sambarino \cite{BPS} and others, and today are recognized as the correct higher rank analogue of convex cocompact subgroups of rank one Lie groups.

Let $\mathbb{K}=\mathbb{R}$ or $\mathbb{C}$. The definition of an Anosov representation into a reductive Lie group $G$ is defined with respect to a pair of opposite parabolic subgroups $(P^{+},P^{-})$ of $G$. Every $P^{\pm}$-Anosov representation $\rho:\Gamma\rightarrow G$ of a word hyperbolic group $\Gamma$ admits a unique pair of $\rho$-equivariant bi-H\"older continuous embeddings of the Gromov boundary of $\Gamma$ (denoted by $\partial_{\infty}\Gamma$) into the homogeneous spaces $G/P^{+}$ and $G/P^{-}$, called the {\em Anosov limit maps of $\rho$}. An Anosov representation into $\mathsf{GL}_{d}(\mathbb{K})$ is called $1$-{\em Anosov} (or {\em projective Anosov}) if it is Anosov with respect to the pair of opposite parabolic subgroups, defined as the stabilizer of a line and a complementary hyperplane in $\mathbb{K}^d$. Guichard--Wienhard (see \cite[Prop. 4.3]{GW}) showed that for every \hbox{$P^{\pm}$-Anosov} representation \hbox{$\rho:\Gamma \rightarrow G$}, there exists $d\in \mathbb{N}$, depending on $G$ and $P^{\pm}$, and a continuous representation $\tau:G \rightarrow \mathsf{GL}_{d}(\mathbb{K})$ such that the representation $\tau\circ \rho:\Gamma \rightarrow \mathsf{GL}_{d}(\mathbb{K})$ is  1-Anosov. The Anosov limit maps of $\tau \circ \rho$ are obtained by composing the limit maps of $\rho$ with the \hbox{$\tau$-equivariant} generalized Pl\"ucker embeddings of $G/P^{+}$ into the projective space $\mathbb{P}(\mathbb{K}^d)$ and of $G/P^{-}$ into the Grassmannian of codimension one hyperplanes $\mathsf{Gr}_{d-1}(\mathbb{K}^d)$. Thus, from this perspective, the class of 1-Anosov representations is the most general class of Anosov representations into real reductive Lie groups.

 The H\"older exponent of a continuous map between two metric spaces is defined as follows. \begin{definition} \label{defexp} Let $(Z,\mathsf{d})$ and $(Z',\mathsf{d}')$ be metric spaces and $f:(Z,\mathsf{d}) \rightarrow (Z',\mathsf{d}')$ a H\"older continuous map. The H\"older exponent of $f$, denoted by $\alpha_{f}(\mathsf{d},\mathsf{d}')$, is defined as follows: $$\alpha_{f}(\mathsf{d},\mathsf{d}'):=\sup \big \{ \alpha>0 \ |\ \exists \  C>0: \mathsf{d}'(f(x),f(y)) \leqslant C \mathsf{d}(x,y)^{\alpha} \ \forall \ x,y \in Z \big\}.$$ \end{definition}

 In this paper, we establish an explicit formula for the H\"older exponent of the Anosov limit maps of a $1$-Anosov representation in terms of the moduli of the eigenvalues of elements in its image and a fixed visual metric on the Gromov boundary of $\Gamma$ (see Theorem \ref{exponent2}). In addition, we also obtain similar formulas for the inverse of the $1$-Anosov limit map of a $(1,1,2)$-hyperconvex representation, as well as we calculate the optimal H\"older exponent for the map conjugating the action of two $(1,1,2)$-hyperconvex representations on their limit sets respectively (see Theorem \ref{conjugation}). Before we state the main results, let us make the following necessary conventions. Throughout this paper, $\Gamma$ denotes a non-elementary hyperbolic group equipped with a $\Gamma$-{\em model space}, i.e. a proper geodesic Gromov hyperbolic space $(X,d_X)$ on which $\Gamma$ acts isometrically, properly discontinuously and cocompactly. Fix $x_0\in X$ and denote by $|\cdot|_{X}:\Gamma \rightarrow \mathbb{N}$ the associated length function defined as $|\gamma|_{X}=d_X(\gamma x_0,x_0)$, $\gamma \in \Gamma$. The {\em stable translation length of $\gamma \in \Gamma$} acting on $X$ is $$|\gamma|_{X,\infty}=\lim_{n\rightarrow \infty}\frac{|\gamma^n|_{X}}{n}.$$ Consider the Gromov product $(\ \ \cdot \ \ )_{x_0}:\partial_{\infty}X \times \partial_{\infty}X \rightarrow [0,\infty]$ and equip $\partial_{\infty}X$ with a visual metric $d_a:\partial_{\infty}X \times \partial_{\infty}X \rightarrow [0,\infty)$, $a>1$, i.e. such that there is $r>1$ (see \cite{Gromov}) with $$r^{-1}a^{-(x\cdot y)_{x_0}}\leq d_a(x,y) \leq r a^{-(x \cdot y)_{x_0}} \ \forall  x,y \in \partial_{\infty}X.$$ 

Let $\mathbb{K}=\mathbb{R}$ or $\mathbb{C}$. For $w \in \mathsf{GL}_d(\mathbb{K})$ denote by $\sigma_1(w)\geq \ldots \geq \sigma_d(w)$ (resp. $\ell_1(w)\geq \ldots \geq \ell_d(w)$) the singular values (resp. moduli of the eigenvalues) of the matrix $w$ in non-increasing order. To simplify notation set $\frac{\sigma_{i}}{\sigma_{j}}(w):=\frac{\sigma_i(w)}{\sigma_{j}(w)}$ and $\frac{\ell_{i}}{\ell_{j}}(w):=\frac{\ell_i(w)}{\ell_{j}(w)}$, $1\leq i ,j \leq d$. The projective space $\mathbb{P}(\mathbb{K}^d)$ is equipped with the metric $d_{\mathbb{P}}$ given by the formula $$d_{\mathbb{P}}\big([v_1],[v_2]\big)=\cos^{-1}\Bigg(\frac{|\langle v_1,v_2\rangle|}{||v_1||\cdot||v_2||}\Bigg) \ \ v_1,v_2\in \mathbb{K}^d\smallsetminus \{{\bf 0}\}.$$

For the precise definition of an Anosov representation into $\mathsf{GL}_d(\mathbb{K})$ see subsection \ref{definitionAnosov}. We recall that a representation $\rho:\Gamma \rightarrow \mathsf{GL}_d(\mathbb{K})$ is called $1$-Anosov if there exist $R,\varepsilon>0$ such that $$\frac{\sigma_1(\rho(\gamma))}{\sigma_2(\rho(\gamma))}\geq e^{\varepsilon|\gamma|_X-R} , \  \ \forall \ \gamma \in \Gamma.$$

Denote by $\Gamma_{\infty}$ the set of infinite order elements of $\Gamma$. Our first main result is a formula for the H\"older exponent of the limit map of a $1$-Anosov representation in terms of the gap between the first and second moduli of eigenvalues of elements in $\rho(\Gamma)$.

\begin{theorem} \label{exponent} Let $\rho:\Gamma \rightarrow \mathsf{GL}_d(\mathbb{K})$ be a $1$-Anosov representation with Anosov limit map \hbox{$\xi_{\rho}^{1}:(\partial_{\infty}X, d_{a}) \rightarrow (\mathbb{P}(\mathbb{K}^d),d_{\mathbb{P}})$}. Let $V:=\langle \xi_{\rho}^1(\partial_{\infty}\Gamma)\rangle$ be the vector subspace spanned by the image of $\xi_{\rho}^1$ and $\rho|_{V}$ the restriction of $\rho$ on $V$. The map $\xi_{\rho}^1$ is H\"older continuous and $$\alpha_{\xi_{\rho}^{1}}(d_{a},d_{\mathbb{P}})=\frac{1}{\log a} \inf_{\gamma \in \Gamma_{\infty}} \frac{\log\ell_1(\rho(\gamma))-\log\ell_2(\rho|_{V}(\gamma))}{|\gamma|_{X,\infty}}.$$ Moreover, if $\rho$ is either irreducible or $2$-Anosov, then $\xi_{\rho}^1$ is $\alpha_{\xi_{\rho}^{1}}(d_{a},d_{\mathbb{P}})$-H\"older. \end{theorem}



We also establish the following double inequality, which is used to prove Theorem \ref{exponent},  for the decay between the second and first singular value of elements in the image of a 1-Anosov representation.

\begin{theorem}\label{semisimple-bound} Let $\rho:\Gamma \rightarrow \mathsf{GL}_d(\mathbb{K})$ be a $1$-Anosov representation. There exists $C>1$ and $m\in \mathbb{Z}$\footnote{For $d\geq 3$ we may take $m=\ell-1$, where $\ell$ is the number of irreducible components of a semisimpification of the second exterior power $\wedge^2\rho:\Gamma \rightarrow \mathsf{GL}(\wedge^2 \mathbb{K}^d)$. In the case where $\rho$ is irreducible then we can take $m=0$.}, $0\leq m \leq \max\big\{0,\frac{d(d-1)}{2}-2\big\}$, with the property that for every $\gamma \in \Gamma$ we have: $$C^{-1}e^{-\beta_{\rho}|\gamma|_X}\leq \frac{\sigma_2(\rho(\gamma))}{\sigma_1(\rho(\gamma))} \leq Ce^{-\alpha_{\rho} |\gamma|_X} |\gamma|_X^{m}$$ where $\alpha_{\rho}:= \underset{\gamma \in \Gamma_{\infty}}{\inf} \frac{\log \frac{\ell_1}{\ell_2} (\rho(\gamma))}{|\gamma|_{X,\infty}}$ and $\beta_{\rho}:= \underset{\gamma \in \Gamma_{\infty}}{\sup}\frac{\log \frac{\ell_1}{\ell_2} (\rho(\gamma))}{|\gamma|_{X,\infty}}$.\end{theorem}

For an irreducible $1$-Anosov representation $\rho$ the expression of the H\"older exponent of $\xi_{\rho}^1$ in Theorem \ref{exponent} is the optimal exponential decay constant of the ratio between the second and first singular value of elements in $\rho(\Gamma)$. Note that in Theorem \ref{semisimple-bound}, where $\rho$ is not necessarily assumed to be semisimple (i.e with reductive Zariski closure), we have a polynomial term of the length function in the upper bound.

As a corollary of Theorem \ref{semisimple-bound}, we obtain the following upper bound for the restriction of the metric $d_{\mathbb{P}}$ on the projective limit set of a $1$-Anosov representation \hbox{in terms of the visual metric $d_a$.}

\begin{corollary}\label{log-bound} Suppose that $\rho:\Gamma \rightarrow \mathsf{GL}_d(\mathbb{K})$ is a $1$-Anosov representation with Anosov limit map \hbox{$\xi_{\rho}^{1}:(\partial_{\infty}X, d_{a}) \rightarrow (\mathbb{P}(\mathbb{K}^d),d_{\mathbb{P}})$}. There exists $C>1$ and $m\in \mathbb{Z}$, $0\leq m\leq \max\big\{0, \frac{d(d-1)}{2}-2\big\}$, with the property that for every pair of distinct points $x,y\in \partial_{\infty}X$ \begin{equation*} d_{\mathbb{P}}\big(\xi_{\rho}^1(x),\xi_{\rho}^1(y)\big)\leq C d_a(x,y)^{\frac{\alpha_{\rho}}{\log a}}\big|\log d_a(x,y)\big|^{m}\end{equation*} where $\alpha_{\rho}>0$ is defined as in Theorem \ref{semisimple-bound}.\end{corollary}

Andr\'es Sambarino in \cite{Sambarino} used the H\"older exponent of an irreducible $1$-Anosov representation \hbox{$\rho:\Gamma \rightarrow \mathsf{PGL}_d(\mathbb{R})$}, where $\Gamma$ is a torsion free hyperbolic group acting convex cocompactly on a complete $\mathsf{CAT}(-1)$ space $(\mathcal{M},d_\mathcal{M})$, in order to establish upper bounds for the Hilbert and spectral entropy of $\rho$. In this setting, in \cite{Sambarino} the following upper bound for the H\"older exponent of $\xi_{\rho}^1$ was proved: $$\alpha_{\xi_{\rho}^{1}}(d_{a},d_{\mathbb{P}}) \leq \frac{1}{\log a}\inf_{\gamma \in \Gamma_{\infty}}\frac{\log \frac{\ell_1}{\ell_2}(\rho(\gamma))}{|\gamma|_{\mathcal{M},\infty}}.$$ Theorem \ref{exponent} shows that the previous upper estimate is in fact an equality.

\subsection{H\"older exponent of $\theta$-Anosov limit maps}\label{generalcase} Let $G$ be a non-compact real algebraic reductive Lie group and $K$ a maximal compact subgroup of $G$. Let also $\mathfrak{a}$ be a Cartan subspace of the Lie algebra $\textup{Lie}(G)$ and $\overline{\mathfrak{a}}^{+}$ a dominant Weyl chamber of $\mathfrak{a}$ so that there is a Cartan decomposition $G=K\exp(\overline{\mathfrak{a}}^{+})K$. Denote by $\mu:G\rightarrow \overline{\mathfrak{a}}^{+}$ the associated Cartan projection. The {\em Lyapunov projection} $\lambda:G\rightarrow \overline{\mathfrak{a}}^{+}$ is defined as follows for $ g\in G$, $$\lambda(g)=\lim_{n\rightarrow \infty}\frac{1}{n}\mu(g^n).$$ 

Fix $\Delta \subset \mathfrak{a}^{\ast}$ a set of simple restricted positive roots and $\theta \subset \Delta$ a subset defining a pair of opposite parabolic subgroups $(P_{\theta}^{+}, P_{\theta}^{-})$ of $G$, see \cite[\S2]{GGKW} for the more background. We fix an irreducible $\theta$-proximal representation $\tau_{\theta}:G\rightarrow \mathsf{GL}_d(\mathbb{K})$ inducing the {\em Pl\"ucker embeddings} $\iota_{\theta}^{+}:G/P_{\theta}^{+} \xhookrightarrow{} \mathbb{P}(\mathbb{R}^d)$ and $\iota_{\theta}^{-}:G/P_{\theta}^{-} \xhookrightarrow{} \mathsf{Gr}_{d-1}(\mathbb{R}^d)$. The homogeneous space $G/P_{\theta}^{\pm}$ is equipped with the metric $d_{\theta^\pm}$, where $d_{\theta^{\pm}}(x,y)=d_{\mathbb{P}}(\iota_{\theta}^{\pm}(x),\iota_{\theta}^{\pm}(y))$ for $x,y\in G/P_{\theta^{\pm}}$. For more background \hbox{we refer to \cite[\S 3 \& \S 5]{GGKW}.}

A representation $\rho:\Gamma \rightarrow G$ is $\theta$-Anosov if and only if $\tau_{\theta}\circ \rho:\Gamma \rightarrow \mathsf{GL}_d(\mathbb{R})$ is $1$-Anosov. In this case, the limit maps of $\tau_{\theta}\circ \rho$ are obtained by precomposing the limit maps of $\rho$ with the Pl\"ucker embedding $\iota_{\theta}^{\pm}$. As a consequence of Theorem \ref{exponent}, for a Zariski dense ${\theta}$-Anosov representation into $G$ we deduce the following general formula for the H$\ddot{\textup{o}}$lder exponent of its Anosov limits maps in terms of the Lyapunov projection on $G$. 

\begin{theorem} \label{Exp-general} Let $G$ be a non-compact real algebraic reductive Lie group and $\theta \subset \Delta$ a subset of simple restricted roots of $G$. Suppose that $\rho:\Gamma \rightarrow G$ is a Zariski dense \hbox{${\theta}$-Anosov representation and} $$\xi_{\rho}^{+}\times \xi_{\rho}^{-}:(\partial_{\infty}X, d_{a}) \rightarrow \big(G/P_{\theta}^{+},d_{\theta^{+}}\big)\times \big(G/P_{\theta}^{-},d_{\theta^{-}}\big)$$ are the Anosov limit maps of $\rho$. Then $\xi_{\rho}^{\pm}$ is $\alpha_{\xi_{\rho}^{\pm}}(d_{a},d_{\theta^{\pm}})$-H\"older and $$\alpha_{\xi_{\rho}^{\pm}}(d_{a},d_{\theta^{\pm}})=\frac{1}{\log a}\inf_{\gamma \in \Gamma_{\infty}} \Bigg(\min_{\varphi\in \theta} \frac{ \varphi(\lambda(\rho(\gamma)))}{|\gamma|_{X,\infty}}\Bigg).$$  \end{theorem}

\subsection{Anosov limit maps not attaining their H\"older exponent.} Denote by $\Sigma_g$ the closed topological surface of genus $g\geq 2$. We equip the hyperbolic plane $\mathbb{H}_{\mathbb{R}}^2$ with the standard hyperbolic Riemannian distance and its Gromov boundary with the visual metric $$d_\textup{v}(x,y)=e^{-(x\cdot y)_{x_0}},  \ x,y\in \partial_{\infty}\mathbb{H}_{\mathbb{R}}^2$$ where $x_0\in \mathbb{H}_{\mathbb{R}}^2$ is a fixed basepoint. As a consequence of Theorem \ref{exponent}, for every $1$-Anosov representation \hbox{$\rho:\Gamma \rightarrow \mathsf{GL}_3(\mathbb{K})$} (which is also $2$-Anosov) and any visual metric $d_a$ on $\partial_{\infty}X$, the limit map $\xi_{\rho}^1:(\partial_{\infty}X,d_a)\rightarrow (\mathbb{P}(\mathbb{K}^3),d_{\mathbb{P}})$ attains its H\"older exponent. However, this is not the case for representations which are neither irreducible nor $\{1,2\}$-Anosov in higher dimension. More precisely, we construct examples in dimension 4:
\begin{theorem}\label{1-Holder} Let \hbox{$\rho_1:\pi_1(\Sigma_g) \rightarrow \mathsf{SL}_2(\mathbb{R})$} be a discrete faithful representation realizing $(\mathbb{H}_{\mathbb{R}}^2,d_{\mathbb{H}^2})$ as a model space for $\pi_1(\Sigma_g)$. There is $\varepsilon>0$ and a continuous family of \hbox{$1$-Anosov} representations $\big\{\rho_{s,t}:\pi_1(\Sigma_g) \rightarrow \mathsf{SL}_4(\mathbb{R})\big\}_{(s,t)\in \mathcal{O}}, \mathcal{O}:=(s,t)\in \mathbb{R}\times (-\varepsilon,\varepsilon),$ such that for $s,t\neq 0$:\\
\noindent \textup{(i)} the image of the $(\rho_1,\rho_{s,t})$-equivariant Anosov limit map $\xi_{s,t}^1:(\partial_{\infty}\mathbb{H}_{\mathbb{R}}^2,d_\textup{v})\rightarrow (\mathbb{P}(\mathbb{R}^4),d_{\mathbb{P}})$ of $\rho_{s,t}$ is spanning $\langle \xi_{s,t}^1(\partial_{\infty}\mathbb{H}_{\mathbb{R}}^2)\rangle =\mathbb{R}^4$,

\noindent \textup{(ii)} $\alpha_{\xi_{s,t}^1}(d_\textup{v},d_{\mathbb{P}})=\frac{1}{2}$ and $\xi_{s,t}^1$ is not $\frac{1}{2}$-H\"older.\end{theorem}

In the previous statement $d_{\mathbb{H}^2}$ denotes the standard Riemannian distance on the real hyperbolic plane $\mathbb{H}_{\mathbb{R}}^2$. For $s,t\neq 0$, the representation $\rho_{s,t}$ in Theorem \ref{1-Holder} is reducible and preserves the $2$-dimensional subspace $W:= \mathbb{R}^2\times \{0\}$, on which the restriction $\rho|_W$ has unipotent image in $\mathsf{GL}(V)$. More precisely, there exists $C_{s,t}>1$ such that for every $\gamma\in \pi_1(\Sigma_g)$, \hbox{$\ell_1(\rho_{s,t}(\gamma))=\ell_1(\rho_1(\gamma))$ and} $$\sigma_2(\rho_{s,t}(\gamma))\leq C_{s,t}\log \sigma_1(\rho_1(\gamma)).$$ 

\subsection{$(1,1,2)$-hyperconvex representations.} The inverse of the limit map of an Anosov represention $\rho:\Gamma \rightarrow \mathsf{GL}_d(\mathbb{K})$ is also H\"older continuous. In a more general setting, we establish that if a representation of $\rho:\Gamma \rightarrow \mathsf{GL}_d(\mathbb{K})$ admits an equivariant continuous injective map $\xi_{\rho}:\partial_{\infty}X \rightarrow \mathbb{P}(\mathbb{K}^d)$, then its inverse is H\"older continuous. 

\begin{theorem} \label{exponent-inverse2} Let $\rho:\Gamma \rightarrow \mathsf{GL}_d(\mathbb{K})$ be a representation admitting a continuous, $\rho$-equivariant, injective map $\xi_{\rho}:(\partial_{\infty}X,d_a) \rightarrow (\mathbb{P}(\mathbb{K}^d),d_{\mathbb{P}})$. The inverse of $\xi_{\rho}$ is H\"older continuous and there exist $c,\beta>0$ such that \begin{equation*} d_{\mathbb{P}}\big(\xi_{\rho}(x),\xi_{\rho}(y)\big)\geq c d_{a}(x,y)^{\beta} \ \ \forall \ x,y\in \partial_{\infty}X.\end{equation*}\end{theorem}

Let us note that a representation $\rho$ admitting an injective limit map $\xi_{\rho}$ might be irreducible but fail to be Anosov with respect to any pair of parabolic subgroups, e.g. see \cite[Ex. 10.1]{Ts20}.

A general formula, similar to the one in Theorem \ref{exponent}, is not known for the inverse of the Anosov limit map of $1$-Anosov representations. However, it is possible to obtain an explicit formula, in terms of the eigenalue of elements, for the particular class of $(1,1,2)$-hyperconvex representations. Pozzetti--Sambarino--Wienhard in \cite{PSW} introduced the notion of $(p,q,r)$-hyperconvex Anosov representation, enjoying common transversality properties with Hitchin representations. We provide here the definition of $(1,1,2)$-hyperconvex representations.

\begin{definition} \textup{(}\textup{Pozzetti--Sambarino--Wienhard} \cite{PSW}\textup{)} Let $\rho:\Gamma \rightarrow \mathsf{GL}_{d}(\mathbb{K})$ be a $\{1,2\}$-Anosov representation and $\xi_{\rho}^{i}:\partial_{\infty}X \rightarrow \mathsf{Gr}_{i}(\mathbb{K}^d)$, $i\in \{1,2,d-2,d-1\}$, the Anosov limit maps of $\rho$. The representation $\rho$ is called $(1,1,2)$-hyperconvex if for every triple $x_1,x_2,x_3 \in \partial_{\infty}X$ of distinct points we have: $$\mathbb{K}^d=\xi_{\rho}^{1}(x_1) \oplus \xi_{\rho}^{1}(x_2)\oplus \xi_{\rho}^{d-2}(x_3).$$\end{definition}

The set of $(1,1,2)$-hyperconvex representations of $\Gamma$ into $\mathsf{GL}_d(\mathbb{K})$ forms an open subset of the space of representations $\textup{Hom}(\Gamma,\mathsf{GL}_d(\mathbb{K}))$ (see \cite[Prop. 6.2]{PSW}). Examples of $(1,1,2)$-hyperconvex representations include Hitchin representations \cite{Labourie} and their exterior powers \cite[Prop. 9.6]{PSW}.

For a representation $\rho:\Gamma \rightarrow \mathsf{GL}_d(\mathbb{K})$ a $\rho$-equivariant map into $\mathbb{P}(\mathbb{K}^d)$ is called {\em spanning} if its image spans $\mathbb{K}^d$. For a $(1,1,2)$-hyperconvex representation with spanning Anosov limit map $\xi_{\rho}^1$, we calculate the H\"older exponent of the inverse in terms of the ratio between the moduli of the first and second eigenvalue of elements.

\begin{theorem} \label{exponent2} Let $\rho:\Gamma \rightarrow \mathsf{GL}_d(\mathbb{K})$ be a $(1,1,2)$-hyperconvex representation with spanning limit map $\xi_{\rho}^1:(\partial_{\infty}X,d_a)\rightarrow (\mathbb{P}(\mathbb{K}^d),d_{\mathbb{P}})$. The H\"older exponent of the inverse $\eta_{\rho}^1:(\xi_{\rho}^1(\partial_{\infty}X),d_{\mathbb{P}})\rightarrow (\partial_{\infty}X,d_a)$ of $\xi_{\rho}^1$ is attained \hbox{and is equal to} $$(\log a)\inf_{\gamma \in  \Gamma_{\infty}}\frac{|\gamma|_{X,\infty}}{\log \frac{\ell_1}{\ell_2}(\rho(\gamma))}.$$ \end{theorem}

Given a $(1,1,2)$-hyperconvex representation $\rho_1:\Gamma \rightarrow \mathsf{GL}_{m}(\mathbb{K})$ and an irreducible $1$-Anosov representation $\rho_2:\Gamma \rightarrow \mathsf{GL}_{r}(\mathbb{K})$, we also obtain a similar formula for the H\"older exponent of the map conjugating the action of $\rho_1(\Gamma)$ on $\xi_{\rho_1}^{1}(\partial_{\infty}X)$ with the action of $\rho_2(\Gamma)$ on $\xi_{\rho_2}^{1}(\partial_{\infty}X)$

\begin{theorem} \label{conjugation} Let $\rho_1:\Gamma \rightarrow \mathsf{GL}_{m}(\mathbb{K})$ be a $(1,1,2)$-hyperconvex representation and $\rho_2:\Gamma \rightarrow \mathsf{GL}_{r}(\mathbb{K})$ a $1$-Anosov representation. Suppose that the limit maps $\xi_{\rho_1}^{1}:\partial_{\infty}X \rightarrow \mathbb{P}(\mathbb{K}^{m})$ and \hbox{$\xi_{\rho_2}^{1}:\partial_{\infty}X \rightarrow \mathbb{P}(\mathbb{K}^{r})$} of $\rho_1$ and $\rho_2$ respectively are spanning. The map conjugating the action of $\rho_1(\Gamma)$ on $\xi_{\rho_1}^1(\partial_{\infty}X)$ with the action of $\rho_2(\Gamma)$ on $\xi_{\rho_2}^1(\partial_{\infty}X)$,  $$\xi_{\rho_2}^{1}\circ (\xi_{\rho_1}^{1})^{-1}:(\xi_{\rho_1}^{1}(\partial_{\infty}X),d_{\mathbb{P}}) \rightarrow (\xi_{\rho_2}^{1}(\partial_{\infty}X),d_{\mathbb{P}}),$$ is H\"older continuous and its exponent is equal to $$\alpha_{\rho_1,\rho_2}:=\underset{\gamma \in \Gamma_{\infty}}{\inf}\frac{\log \frac{\ell_1}{\ell_2}(\rho_2(\gamma))}{\log \frac{\ell_1}{\ell_2}(\rho_1(\gamma))}.$$ Moreover, if $\rho_2$ is either irreducible or $\{1,2\}$-Anosov, then $\xi_{\rho_2}^{1}\circ (\xi_{\rho_1}^{1})^{-1}$ is $\alpha_{\rho_1,\rho_2}$-H\"older. \end{theorem}

Recently, Theorem \ref{conjugation} was used by Alex Nolte in order to define asymmetric metrics with interesting completeness properties \cite[Thm. B, C \& D]{Nolte} on the Hitchin component of a closed surface group.

\subsection{Related notions} The H\"older regularity of the projective limit set of particular classes of Anosov representations was studied in \cite{Guichard, Zhang-Zimmer}. For a strictly convex domain $\Omega\subset \mathbb{P}(\mathbb{R}^d)$ and a discrete group $\Gamma<\mathsf{PGL}_d(\mathbb{R})$ preserving and acting cocompactly on $\Omega$, Guichard in \cite[Thm. 11]{Guichard} established an explicit expression for the optimal H\"older regularity and $\beta$-convexity of the boundary $\partial \Omega$ in terms of the eigenvalue data of the dividing group\footnote{The inclusion $\Gamma \xhookrightarrow{} \mathsf{SL}_d(\mathbb{R})$ of the dividing group acting cocompactly on $\Omega$ is $1$-Anosov by \cite{Benoist-div}.}. Following \cite[Def. 7]{Guichard}, the notion of optimal $\beta$-convexity for the boundary $\partial\Omega$ is defined as follows \begin{equation}\label{beta-convex}\beta_{\Omega}:=\inf \Big \{\beta>0 \ | \ \exists \ C>0:  d_{\mathbb{P}}(x,y)\leq C\textup{dist}_{\mathbb{P}(\mathbb{K}^d)}(x,\xi^{-}(y))^{\frac{1}{\beta}} \ \forall \ x,y\in \partial\Omega\Big\}\end{equation} where $\xi^{-}(y)=T_{y}\partial\Omega$ is the tangent hyperplane associated to $y\in \partial \Omega$ and $\xi^{-}:\partial\Omega \rightarrow \mathsf{Gr}_{d-1}(\mathbb{R}^d)$ is the Anosov limit map of the $1$-Anosov subgroup $\Gamma<\mathsf{PGL}_d(\mathbb{R})$. Guichard proved the following formula in \cite[Thm. 22 (2)]{Guichard}: \begin{equation}\label{beta-eigenvalue}\beta_{\Omega}=\sup_{\gamma \in \Gamma_{\infty}}\frac{\log \frac{\ell_1}{\ell_d}(\gamma)}{\log \frac{\ell_1}{\ell_2}(\gamma)}.\end{equation} Let $d_{\Omega}$ be the Hilbert metric on $\Omega$ and $d_a$, $a>0$, a visual metric on the Gromov boundary $\partial \Omega$ of the  hyperbolic space $(\Omega,d_{\Omega})$, such that $d_{a}(x,y)\asymp a^{-(x\cdot y)_{\Omega}}$ and $(\ \ \cdot \ \ )_{\Omega}$ is the Gromov product with respect to $d_{\Omega}$. By using the calculation in \cite[Prop. 3.6.2]{thesis}, there is $c>1$ with the property: $$c^{-1}e^{-2(x\cdot y)}\leq \textup{dist}_{\mathbb{P}(\mathbb{K}^d)}(x,\xi^{-}(y))\textup{dist}_{\mathbb{P}(\mathbb{K}^d)}(y,\xi^{-}(x))\leq ce^{-2(x \cdot y)},\ \forall \ x,y\in \partial\Omega.$$ Hence, as a consequence of the previous estimate, the definition (\ref{beta-convex}) and Definition \ref{defexp}, we have $$\beta_{\Omega}^{-1}=(\log a)\alpha_{\xi}(d_{a},d_{\mathbb{P}}),$$ where $\xi:(\partial\Omega,d_a)\rightarrow (\partial \Omega,d_{\mathbb{P}})$ is the identity map (which coincides with the limit map of the inclusion $\Gamma\xhookrightarrow{} \mathsf{PGL}_d(\mathbb{R})$). Therefore, the formula for $\beta_{\Omega}$ in (\ref{beta-eigenvalue}) coincides with the formula provided by Theorem \ref{exponent} for the model space $(\Omega,d_{\Omega})$ of $\Gamma$ and the visual metric $d_a$ on $\partial \Omega$.

The optimal H\"older regularity for $\partial\Omega$ in the definition \cite[Def. 5]{Guichard} coincides with the H\"older exponent of the identity map $\xi^{-1}:(\partial\Omega,d_{\mathbb{P}})\rightarrow (\partial \Omega,d_{a})$, seen as the inverse of the Anosov limit map $\xi$. The formula in \cite[Thm. 22 (1)]{Guichard} provides the optimal \hbox{value of the H\"older exponent of $\xi^{-1}$.}
 
\par  Zhang--Zimmer in \cite{Zhang-Zimmer} established conditions under which the proximal limit set of an $1$-Anosov representation is a $C^{1+\alpha}$-submanifold of the ambient real projective space and provided a formula for the optimal value of $\alpha>0$ in terms of the eigenvalue data of the representation \cite[Thm. 1.12 \& Thm. 1.14]{Zhang-Zimmer}. In the case where the image of an $1$-Anosov representation $\rho(\Gamma)$ preserves a strictly convex domain $\Omega\subset \mathbb{P}(\mathbb{R}^d)$), they calculate the optimal value $\alpha>0$ such that the boundary $\partial \Omega$ is of class $C^{1+\alpha}$ along $\xi_{\rho}^1(\partial_{\infty}\Gamma)\subset \partial \Omega$, generalizing the result from \cite[Thm. 22 (2)]{Guichard}.


\medskip

\noindent \textbf{Acknowledgements.} I would like to thank Richard Canary, Sami Douba, Fran\c{c}ois Gu\'eritaud, Fanny Kassel, Alex Nolte, Rafael Potrie, Nicolas Tholozan and Anna Wienhard for interesting discussions, as well as Gregorio Baldi for his help with tex files. I would also like to to thank the referee for carefully reading the paper and their comments and suggestions. This project received funding from the European Research Council (ERC) under the European's Union Horizon 2020 research and innovation programme (ERC starting grant DiGGeS, grant agreement No 715982). 

\section{Background} \label{background}

\subsection{Gromov products.} Let $(X,d_X)$ be a proper geodesic metric space. The {\em Gromov product with respect to $x_0\in X$} is the map \hbox{$(\  \ \cdot \ \ )_{x_0}:X\times X \rightarrow [0,\infty)$} defined as follows: $$(x\cdot y)_{x_0}:=\frac{1}{2}\big(d_X(x,x_0)+d_X(y,y_0)-d_X(x,y)\big) \ \  x,y \in X.$$ For an isometry $\gamma:X \rightarrow X$ we set $|\gamma|_{X}:=d_X(\gamma x_0,x_0)$. The {\em stable translation length of $\gamma$} is $$|\gamma|_{X, \infty}=\lim_{n \rightarrow \infty}\frac{|\gamma^n|_{X}}{n}.$$  The metric space $(X,d_X)$ is called $\delta$-{\em Gromov hyperbolic} (see \cite{Gromov}) if $$(x\cdot y)_{x_0}\geq \min \big\{ (x \cdot z)_{x_0},(z \cdot y)_{x_0}\big \}-\delta \ \ \forall \ x_0,x,y,z \in X.$$ In this case the Gromov product extends to a map on the Gromov boundary of $X$ (denoted by $\partial_{\infty}X$), $( \ \cdot \ )_{x_0}:\partial_{\infty}X\times \partial_{\infty}X\rightarrow [0,\infty]$, defined as follows $$(x\cdot y)_{x_0}:=\sup \big\{\varliminf_{n,m \rightarrow \infty} (x_m \cdot y_n )_{x_0}\ \big |\ x=\lim_{m \rightarrow \infty} x_m, \ y=\lim_{n \rightarrow \infty} y_n \big\}.$$ Moreover, the boundary $\partial_{\infty}X$ is a compact metrizable space, e.g. see \cite{Gromov} and  \cite[Prop. 3.21, III]{BH}: there exist $r,a>1$ and a {\em visual metric} $d_a:\partial_{\infty}X\times \partial_{\infty}X \rightarrow [0,\infty)$ satisfying \begin{equation} \label{visualineq} r^{-1}a^{-(x\cdot y)_{x_0}}\leq d_a(x,y) \leq r a^{-(x\cdot y)_{x_0}} \ \ \forall \ x,y \in \partial_{\infty}X.\end{equation}

\noindent {\bf Convention:} Throughout this paper, $\Gamma$ is assumed to be a non-elementary word hyperbolic group. From now on, $(X,d_X)$ denotes a proper geodesic $\Gamma$-model space, equipped with a proper discontinuous, isometric and cocompact action by $\Gamma$. Given $x_0\in X$, the orbit map $(\Gamma,d_{\Gamma}) \rightarrow (X,d_X)$, $\gamma \mapsto \gamma x_0$, extends to a $\Gamma$-equivariant homeomorphism\footnote{In fact a bi-H\"older homeomorphism after fixing visual metrics on $\partial_{\infty}\Gamma$ and $\partial_{\infty}X$.} $\partial_{\infty}\Gamma \cong \partial_{\infty}X$. We denote by $\partial_{\infty}^{(2)}X$ (resp. $\partial_{\infty}^{(3)}X$) the set of 2-tuples (resp. 3-tuples) of distinct points of $\partial_{\infty}X$. The group $\Gamma$ acts properly discontinuously and cocompactly on $\partial_{\infty}^{(3)}X$ (e.g. see \cite{Gromov, BH}) and fix $\mathcal{F} \subset \partial_{\infty}^{(3)}X$ a compact subset such that $$\partial_{\infty}^{(3)}X=\bigcup_{\gamma \in \Gamma}\gamma \mathcal{F}.$$

We will need the following folklore lemma for the Gromov product.

\begin{lemma} \label{Gromovproduct1} Let $(X,d_X)$ be a $\delta$-Gromov hyperbolic model space for $\Gamma$ and fix $x_0 \in X$. For every $x,y \in \partial_{\infty}X$ and $\gamma \in \Gamma$ the following inequality holds $$\big|(\gamma x\cdot \gamma y)_{x_0}-|\gamma |_{X}-(x\cdot y)_{x_0}+(x\cdot \gamma^{-1}x_0)_{x_0}+(y \cdot \gamma^{-1}x_0)_{x_0} \big| \leq 4\delta.$$ \end{lemma}

\begin{proof} Let us choose sequences $(x_n)_{n\in \mathbb{N}}, (y_n)_{n\in \mathbb{N}}\subset X$ such that $\lim_n x_n=x$, $\lim_n y_n=y$ and $\lim_n (\gamma x_n\cdot \gamma y_n)_{x_0}=(\gamma x\cdot \gamma y)_{x_0}$. By Gromov hyperbolicity, if necessary after passing to a subsequence, we have that $$\varlimsup_{n \rightarrow \infty}\big(|(x\cdot \gamma^{-1}x_0)_{x_0}-(x_n \cdot \gamma^{-1}x_0)_{x_0}|+ |(y\cdot \gamma^{-1}x_0)_{x_0}-(y_n \cdot \gamma^{-1}x_0)_{x_0}|+ |(x\cdot y)_{x_0}-(x_n \cdot y_n)_{x_0}|\big)\leq 3\delta.$$
The lemma now follows by observing that for every $n\in \mathbb{N}$ we have the equality $$(\gamma x_n\cdot \gamma y_n)_{x_0}=|\gamma|_{X}+(x_n\cdot y_n)_{x_0}-(x_n\cdot \gamma^{-1}x_0)_{x_0}-(y_n\cdot \gamma^{-1}x_0)_{x_0}.$$  \end{proof}

For a hyperbolic isometry $\gamma:X\rightarrow X$ we denote by $\gamma^{+}\in \partial_{\infty}X$ the unique attracting fixed point of $\gamma$ in the Gromov boundary of $X$.

\begin{lemma} \label{Gromovproduct2} Let $(X,d_X)$ be a $\delta$-Gromov hyperbolic space and $\gamma:X\rightarrow X$ be a hyperbolic isometry. Fix $x_0 \in X$. Then the following inequality holds: $$\big|2(\gamma^{+}\cdot \gamma^{-1}x_0)_{x_0}-\big(|\gamma|_{X}-|\gamma|_{X,\infty}\big) \big|\leq 2\delta.$$\end{lemma}

\begin{proof} By assumption, for every $x,y,z \in X \cup \partial_{\infty}X$ we have $$(x\cdot y)_{x_0} \geqslant \min \big \{(x\cdot z)_{x_0}, (z \cdot y)_{x_0} \big \}-\delta.$$ Since $|\gamma|_{X,\infty}=\lim_n\frac{|\gamma^n|_X}{n}$, it is not hard to check that $$\varliminf_{n \rightarrow \infty}\big(|\gamma^{n+1}|_{X}-|\gamma^n|_{X}\big)\leqslant|\gamma|_{X,\infty}\leqslant \varlimsup_{n \rightarrow \infty}\big(|\gamma^{n+1}|_{X}-|\gamma^n|_{X}\big)$$ and fix a sequence $(k_n)_{n \in \mathbb{N}}\subset \mathbb{N}$ with $\varlimsup_{n}(|\gamma^{n+1}|_{X}-|\gamma^n|_{X})=\lim_{n}(|\gamma^{k_n+1}|_{X}-|\gamma^{k_n}|_{X})$. Now note that $\lim_n ((\gamma^{+} \cdot \gamma^{k_n}x_0)_{x_0}-(\gamma^{+} \cdot \gamma^{-1}x_0)_{x_0})=\infty$, so for large $n$ we deduce \begin{align*} \big(\gamma^{+}\cdot \gamma^{-1}x_0 \big)_{x_0} & \leqslant \frac{1}{2}\lim_{n \rightarrow \infty}\big(|\gamma|_{X}+\big|\gamma^{k_n}\big|_{X}-\big|\gamma^{k_n+1}\big|_{X}\big)+\delta \leqslant \frac{1}{2} \big( |\gamma|_{X}-|\gamma|_{X,\infty}\big)+\delta.\end{align*} Similarly, if $(m_n)_{n \in \mathbb{N}}$ is a sequence with \hbox{$\varliminf_{n}\big(|\gamma^{n+1}|_{X}-|\gamma^n|_{X}\big)=\lim_{n}\big(|\gamma^{m_n+1}|_{X}-|\gamma^{m_n}|_{X}\big)$, then} \begin{align*} \big(\gamma^{+}\cdot \gamma^{-1}x_0 \big)_{x_0}& \geqslant \lim_{n\rightarrow \infty}\big(\gamma^{m_n}x_0\cdot \gamma^{-1}x_0\big)_{x_0}-\delta \geqslant \frac{1}{2} \big(|\gamma|_{X}-|\gamma|_{X,\infty}\big)-\delta.\end{align*} The inequality follows and the proof of the lemma is complete. \end{proof}

Denote by $B_{\epsilon}(w_0)=\{y\in \partial_{\infty}X: d_a(y,w_0)<\epsilon\}$ the open ball of radius $\epsilon>0$ centered at $w_0\in \partial_{\infty}X$ with respect to $d_a$. We will need the following lemma.

\begin{lemma}\label{triple} Let $(X,d_X)$ be a $\delta$-Gromov hyperbolic model space for $\Gamma$. Fix $z,w_0\in \partial_{\infty}X$ with $z\neq w_0$. There exist $C, \varepsilon>0$, depending on the choice of $z,w_0\in \partial_{\infty}X$, with the property: if $x,y\in B_{\varepsilon}(w_0)$ are distinct and $\gamma\in \Gamma$ is an element such that $(x,y,z)=\gamma(b_1,b_2,b_3)$ for some $(b_1,b_2,b_3)\in \mathcal{F},$ then \begin{align*}\big|(x\cdot y)_{x_0}-|\gamma|_{X} \big|\leq C, \ \big|(b_3\cdot \gamma^{-1}x_0)_{x_0}-|\gamma|_{X}\big|\leq C.\end{align*} \end{lemma}

\begin{proof} Let $r,a>1$ satisfying (\ref{visualineq}) and set $\varepsilon:=\frac{1}{2}d_a(w_0,z)$. For every $x\in B_{\varepsilon}(w_0)$ we have that $$ra^{-(x\cdot z)_{x_0}}\geq d_a(x,z)\geq \frac{1}{2}d_a(w_0,z).$$ In particular, for every $x\in B_{\varepsilon}(w_0)$ we have $$(x\cdot z)_{x_0}\leq\frac{1}{\log a}\log \frac{2r}{d_a(w_0,z)}.$$ Let us also set $\mathcal{D}:=\max\big\{\max \big\{(x_{i}\cdot x_{j})_{x_0}: (x_1,x_2,x_3)\in \mathcal{F}, i\neq j\big\}, \frac{1}{\log a}\log \frac{2r}{d_a(w_0,z)}\big\}.$
Since $\Gamma$ acts properly on $\partial_{\infty}^{(3)}X$, by shrinking $\varepsilon>0$ if necessary, we may assume that every $g\in \Gamma$ with the property that $(x,y,z)\in g\mathcal{F}$, for some $x,y\in B_{\varepsilon}(w_0)$, satisfies $|g|_{X}>7\delta+4\mathcal{D}$.
\hbox{Observe that} $$\min\big \{(b_1 \cdot \gamma^{-1}x_0)_{x_0},(b_2 \cdot \gamma^{-1}x_0)_{x_0}\big \}\leq (b_1\cdot b_2)_{x_0}+\delta \leq \mathcal{D}+\delta$$ and without loss of generality we may assume $(b_1\cdot \gamma^{-1}x_0)_{x_0}\leq \mathcal{D}+\delta$. By using Lemma \ref{Gromovproduct1} we obtain the following inequalities: \begin{align*} \big|(x\cdot y)_{x_0}-|\gamma|_{X}-(b_1\cdot b_2)_{x_0}+(b_1\cdot \gamma^{-1}x_0)_{x_0}+(b_2\cdot \gamma^{-1}x_0)_{x_0}\big|&\leq 4\delta \\ \big|(x\cdot z)_{x_0}-|\gamma|_{X}-(b_1\cdot b_3)_{x_0}+(b_1 \cdot \gamma^{-1}x_0)_{x_0}+(b_3 \cdot \gamma^{-1}x_0)_{x_0}\big|&\leq 4\delta. \end{align*} Note that since $\max\big \{(b_1\cdot b_3)_{x_0},(x\cdot z)_{x_0}\big \}\leq \mathcal{D}$ and $(b_1\cdot \gamma^{-1}x_0)_{x_0}\leq \delta+\mathcal{D}$, the second inequality shows that $$\big|(b_3 \cdot \gamma^{-1}x_0)_{x_0}-|\gamma|_{X}\big|\leq (b_1\cdot b_3)_{x_0}+(b_1\cdot \gamma^{-1}x_0)_{x_0}+(x\cdot z)_{x_0}+4\delta\leq 5\delta+3\mathcal{D}.$$ In particular, $(b_3 \cdot \gamma^{-1}x_0)_{x_0}>2\delta+\mathcal{D}$. Furthermore, observe that $$\min\big \{(b_2 \cdot \gamma^{-1}x_0)_{x_0},(b_3\cdot \gamma^{-1}x_0)_{x_0}\big \}\leq (b_2\cdot b_3)_{x_0}+\delta \leq \delta+\mathcal{D},$$ hence we necessarily have $(b_3\cdot \gamma^{-1}x_0)_{x_0}>(b_2\cdot \gamma^{-1}x_0)_{x_0}$ and $(b_2\cdot \gamma^{-1}x_0)_{x_0}\leq \delta+\mathcal{D}$. Finally, we conclude the estimate $$\big|(x\cdot y)_{x_0}-|\gamma|_{X}\big|\leq (b_1\cdot b_2)_{x_0}+(b_1 \cdot \gamma^{-1}x_0)_{x_0}+(b_2\cdot \gamma^{-1}x_0)_{x_0}+ 4\delta\leq 6\delta+3\mathcal{D}.$$ The statement of the lemma holds true by considering $C:=6\delta+3\mathcal{D}$. \end{proof}

\subsection{Cartan decomposition} Let $\mathbb{K}=\mathbb{R},\mathbb{C}$ and $(e_1, \ldots ,e_d)$ the canonical basis of $\mathbb{K}^d$, where $e_i$ is the vector all of whose cordinates are $0$ except with a $1$ on the $i$-th coordinate. Denote by $\langle \cdot,\cdot \rangle$ the standard Hermitian inner product on $\mathbb{K}^d$ and by $\mathsf{K}_d$, where $\mathsf{K}_d=\mathsf{O}(d)$ when $\mathbb{K}=\mathbb{R}$ and $\mathsf{K}_d=\mathsf{U}(d)$ when $\mathbb{K}=\mathbb{C}$, the corresponding maximal compact subgroup of $\mathsf{GL}_d(\mathbb{K})$ to the choice of the inner product.

For a matrix $g \in \mathsf{GL}_d(\mathbb{K})$ we denote by $\sigma_1(g) \geq \ldots \geq \sigma_d(g)$ the singular values of $g$ in non-increasing order and by $\ell_1(g) \geq \ldots \geq \ell_d(g)$ the moduli of the eigenvalues of $g$ in non-increasing order. We recall that $\sigma_i(g)=\sqrt{\ell_i(gg^{\ast})}$ for $1 \leq i \leq d$, where $g^\ast$ denotes the conjugate transpose matrix of $g$. The standard Cartan decomposition of $\mathsf{GL}_{d}(\mathbb{K})$ is \begin{equation*} \mathsf{GL}_d(\mathbb{K})=\mathsf{K}_d \exp \big( \textup{diag}^{+}(d) \big) \mathsf{K}_d\end{equation*} where $\textup{diag}^{+}(d)=\big \{\textup{diag}\big(a_1, \ldots, a_{d} \big): a_1 \geq \ldots \geq a_d \big\}.$ The Cartan projection is the continuous, proper and surjective map $\mu: \mathsf{GL}_d(\mathbb{K}) \rightarrow \textup{diag}^{+}(d),$ $$\mu(g)=\big(\log \sigma_1(g), \ldots, \log \sigma_d(g)\big).$$

\subsection{Anosov representations} \label{definitionAnosov} We use the following theorem as a definition for $k$-Anosov representations, $1\leq k\leq d-1$, into $\mathsf{GL}_d(\mathbb{K})$. For the dynamical definition of an Anosov representation into linear semisimple Lie groups we refer the reader to \cite{Labourie} and \cite[Def. 2.10]{GW}.

\begin{theorem} \label{Anosov} \textup{(}\cite{BPS,KLP2,KP}\textup{)} Let $\mathsf{\Gamma}$ be a finitely generated group and $|\cdot|_{\mathsf{\Gamma}}:\mathsf{\Gamma} \rightarrow \mathbb{N}$ a word length function induced by some finite generating subset of $\mathsf{\Gamma}$. Suppose that $\rho:\mathsf{\Gamma} \rightarrow \mathsf{GL}_d(\mathbb{K})$ is a representation and fix $1 \leq k \leq d-1$. Then the following conditions are equivalent:
\medskip

\noindent \textup{(i)} $\mathsf{\Gamma}$ is word hyperbolic and $\rho$ is ${k}$-Anosov.\\
\noindent \textup{(ii)} There exist $R,\varepsilon>0$ such that $\frac{\sigma_{k}}{\sigma_{k+1}}(\rho(\gamma)) \geq e^{\varepsilon|\gamma|_{\mathsf{\Gamma}}-R}$ for every $\gamma \in \mathsf{\Gamma}$.\\
\noindent \textup{(iii)} $\mathsf{\Gamma}$ is word hyperbolic and there exists $c>0$ such that $ \frac{\ell_{k}}{\ell_{k+1}}(\rho(\gamma))  \geq e^{c|\gamma|_{\mathsf{\Gamma},\infty}}$ for every $\gamma \in \mathsf{\Gamma}$.\end{theorem}

\noindent The equivalence $\textup{(i)} \Leftrightarrow \textup{(ii)}$ was established by Kapovich--Leeb--Porti \cite{KLP2} and independently by Bochi--Potrie--Sambarino \cite{BPS}\footnote{While the results from \cite{KLP2, BPS} are stated for representations into semisimple Lie groups, note that $\rho$ is $k$-Anosov if and only in $\hat{\rho}:\mathsf{\Gamma}\rightarrow \mathsf{SL}^{\pm}_d(\mathbb{K})$, $\hat{\rho}(\gamma)=|\textup{det}\rho(\gamma)|^{-1/d}\rho(\gamma)$, is k-Anosov, since \hbox{$\frac{\sigma_k}{\sigma_{k+1}}(\rho(\gamma))=\frac{\sigma_k}{\sigma_{k+1}}(\hat{\rho}(\gamma))$ for every $\gamma$.}}, while the implication $\textup{(iii)} \Rightarrow \textup{(ii)}$ was established by Kassel--Potrie in \cite{KP}. Every $k$-Anosov representation $\rho:\Gamma \rightarrow \mathsf{GL}_d(\mathbb{K})$ admits a unique pair of continuous, $\rho$-equivariant maps $$(\xi_{\rho}^{k}, \xi_{\rho}^{d-k}):\partial_{\infty}X  \rightarrow\mathsf{Gr}_{k}(\mathbb{K}^d)\times \mathsf{Gr}_{d-k}(\mathbb{K}^d)$$ called the {\em $k$-Anosov limit maps} of $\rho$. We summarize here some of their main properties:
\medskip

\noindent \textup{(i)} $\xi_{\rho}^k$ and $\xi_{\rho}^{d-k}$ are {\em compatible,} i.e. $\xi_{\rho}^k(x) \subset \xi_{\rho}^{d-k}(x)$ for every $x\in \partial_{\infty}X.$\\
\noindent \textup{(ii)} $\xi_{\rho}^k$ and $\xi_{\rho}^{d-k}$ are {\em transverse,} i.e. $\mathbb{K}^d=\xi_{\rho}^{k}(x)\oplus \xi_{\rho}^{d-k}(y)$ for every $x,y \in \partial_{\infty}X$ with $x \neq y$.\\
\noindent \textup{(iii)} $\xi_{\rho}^k$ and $\xi_{\rho}^{d-k}$ are {\em dynamics preserving}: for every $\gamma \in \Gamma$ of infinite order, $\xi_{\rho}^{k}(\gamma^{+})$ (resp. $\xi_{\rho}^{d-k}(\gamma^{+})$) is the attracting fixed point of $\rho(\gamma)$ in $\mathsf{Gr}_k(\mathbb{K}^d)$ (resp. $\mathsf{Gr}_{d-k}(\mathbb{K}^{d}))$.
\medskip

For more background on Anosov representations and their properties we also refer the reader to \cite{Canary, GGKW, GW, KLP2, Labourie}. 

\subsection{Approximating singular values by eigenvalues.} Let $\mathsf{\Gamma}$ be an abstract group. A representation $\psi:\mathsf{\Gamma} \rightarrow \mathsf{GL}_d(\mathbb{K})$ is called {\em semisimple} if $\psi$ decomposes as a direct sum of irreducible representations. In this case, the Zariski closure $\overline{\psi(\mathsf{\Gamma})}^{\textup{Zar}}$ of $\psi(\mathsf{\Gamma})$ in $\mathsf{GL}_{d}(\mathbb{K})$ is a real reductive algebraic Lie group. 

The following result from \cite{benoist-limitcone}, based on work of Abels--Margulis--Soifer \cite{AMS}, offers a connection between eigenvalues and singular values of elements in the image of a semisimple representation. 

\begin{theorem} \textup{(Abels--Margulis--Soifer \cite{AMS}, Benoist \cite{benoist-limitcone})} \label{finitesubset} Suppose that $\big\{\rho_{i}:\mathsf{\Gamma} \rightarrow \mathsf{GL}_{d_i}(\mathbb{K})\big\}_{i=1}^{q}$ is a finite collection of semisimple representations. There exists a finite subset $F\subset \mathsf{\Gamma}$ and $C>0$ with the property: for every $\delta \in \mathsf{\Gamma}$ there exists $f \in F$ such that for every $1\leq i\leq q$ we have $$\max_{1 \leq j \leq d_i}\big|\log \sigma_j (\rho_i(\delta))-\log \ell_j(\rho_i(\delta f))\big|\leq C.$$ \end{theorem}

For a proof of the previous theorem we refer the reader to \cite[Thm. 4.12]{GGKW}. Kassel--Potrie in \cite{KP1} established the following strengthened version of Theorem \ref{finitesubset} for semisimple representations of semigroups and associated actions on hyperbolic spaces. We will use the following corollary of their main theorem.

\begin{theorem}\textup{(Kassel--Potrie \cite[Cor. 1.8]{KP1})} \label{AMS-KP}  Let $\Gamma$ be a semigroup acting by isometries on a direct product $\mathcal{M}$ of finitely many Gromov hyperbolic metric spaces, such that the action on none of the factors has a unique global fixed point at infinity. For any Euclidean space $V$ and any semisimple representation $\rho:\Gamma \rightarrow \mathsf{GL}(V)$, there exist $C > 0$ and a finite subset
$F$ of $\Gamma$ with the following property: for every $\gamma \in \Gamma$ there exists $f \in F$ such that \begin{align*} \big||\gamma|_{\mathcal{M}}-|\gamma f|_{\mathcal{M},\infty} \big|&\leq C \\ \max_{1 \leq j \leq d_i}\big|\log \sigma_j (\rho_i(\gamma))-\log \ell_j(\rho_i(\gamma f))\big|&\leq C.\end{align*}\end{theorem}

As a consequence of Theorem \ref{AMS-KP} and the strong proximality properties of Anosov representations (see \cite[Thm. 1.7 (iv)]{GW}) we obtain the following proposition which we use for the proof of Theorem \ref{exponent}, Theorem \ref{semisimple-bound}, Theorem \ref{exponent2} and Theorem \ref{conjugation}. For the notion of the semisimplification of a linear representation of a group $\mathsf{\Gamma}$, which we use below, we refer the reader to \cite[\S 2]{GGKW}.

\begin{proposition} \label{12finitesubset} Let $\Gamma$ be a word hyperbolic group. Suppose that $\rho_1:\Gamma \rightarrow \mathsf{GL}_d(\mathbb{K})$ is a $\{1,2\}$-Anosov representation and $\rho_2:\Gamma \rightarrow \mathsf{GL}_{m}(\mathbb{K})$ is a semisimple representation. Then there exists a finite subset $F\subset \Gamma$ and $R>0$ with the property for every $\gamma \in \Gamma$ there exists $f \in F$: \begin{align*} \big||\gamma|_{X}-|\gamma f|_{X,\infty}\big|&\leq R,\\ \big|\log \sigma_1 (\rho_1(\gamma))-\log \ell_1(\rho_1(\gamma f))\big|&\leq R,\\ \big|\log \sigma_2 (\rho_1(\gamma))-\log \ell_2(\rho_1(\gamma f))\big|&\leq R,\\ \max_{1 \leq i \leq m} \big|\log \sigma_i (\rho_2(\gamma))-\log \ell_i (\rho_2(\gamma f))\big|&\leq R.\end{align*} \end{proposition}

\begin{proof} Let $\rho_1^{ss}:\Gamma \rightarrow \mathsf{GL}_d(\mathbb{K})$ be a semisimplification of $\rho_1$. Since $\rho_1$ is $\{1,2\}$-Anosov, by applying \cite[Lem. 2.10]{Ts20} for $\rho_1$ and $\wedge^2 \rho_1$, there exists $D>0$ such that for every $\gamma \in \Gamma$ we have $$\max_{i=1,2}\big|\log \sigma_{i}(\rho_1(\gamma))-\log \sigma_i (\rho_1^{ss}(\gamma))\big|\leq D.$$ Now note that since $\Gamma$ acts cocompactly on $(X,d_X)$, by the stability of geodesics in Gromov hyperbolic spaces \cite[Ch. III.H]{BH}, there exists $R_1>0$ such that for every $\gamma \in \Gamma$ we have: $$(\gamma^{+}\cdot \gamma^{-1}x_0)_{x_0}\leq R_1(\gamma^{+}\cdot \gamma^{-1})_{e}+R_1.$$ By Lemma \ref{Gromovproduct2} we conclude that there exist $R_2>1$ such that for every $\gamma \in \Gamma$: $$|\gamma|_{X}-|\gamma|_{X,\infty} \leq R_1\big(|\gamma|_{\Gamma}-|\gamma|_{\Gamma,\infty}\big)+R_2.$$The conclusion follows by applying Theorem \ref{AMS-KP} for the $\Gamma$-model space $(X,d)$ and the semisimple representation \hbox{$\rho_{1}^{ss}\times \rho_2:\Gamma \rightarrow \mathsf{GL}(\mathbb{K}^d\oplus \mathbb{K}^m)$.}\end{proof}

\section{Some Lemmata}\label{lemmata} 
In this section, we provide some lemmata for the distance between points in the limit set of an $1$-Anosov representation. We recall that equip the projective space $\mathbb{P}(\mathbb{K}^d)$ with the metric $d_{\mathbb{P}}$ defined by the formula $$d_{\mathbb{P}}\big([v_1],[v_2]\big)=\cos^{-1}\Bigg(\frac{|\langle v_1,v_2\rangle|}{||v_1||\cdot||v_2||}\Bigg), \ \ v_1,v_2\in \mathbb{K}^d\smallsetminus \{{\bf 0}\}.$$ \par Let $(X,d_X)$ be a $\Gamma$-model space and $\mathcal{F}\subset \partial_{\infty}^{(3)}X$ a compact subset with $\partial^{(3)}_{\infty}X=\bigcup_{\gamma \in \mathcal{F}}\gamma \mathcal{F}$. Recall from \cite{PSW} that a representation $\rho:\Gamma \rightarrow \mathsf{GL}_{d}(\mathbb{K})$ is called {\em $(1,1,2)$-hyperconvex} if $\rho$ is $\{1,2\}$-Anosov and for every triple of distinct points $(x_1,x_2,x_3) \in \partial_{\infty}^{(3)}X$ we have $$\mathbb{K}^d=\xi_{\rho}^{1}(x_1) \oplus \xi_{\rho}^{1}(x_2)\oplus \xi_{\rho}^{d-2}(x_3).$$ 

Estimates (i) and (ii) in the following lemma are a consequence of \cite[Lem. 5.3]{PSW} and \cite[Prop. 5.7]{PSW} respectively.

\begin{lemma} \label{mainlemma} Let $\Gamma$ be a non-elementary word hyperbolic group, $w_0\in \partial_{\infty}X$ and  $z \in \partial_{\infty}X \smallsetminus \{w_0\}$. There exists $\varepsilon>0$, depending only on $z,w_0\in \partial_{\infty}X$ and $\Gamma$, with the property: if $x,y\in B_{\varepsilon}(w_0)$ are distinct and $(x,y,z)=\gamma(b_1,b_2,b_3)$ for some $(b_1,b_2,b_3)\in \mathcal{F}$ and $\gamma \in \Gamma$, then:
\medskip

\noindent \textup{(i)} for every $1$-Anosov representation $\rho:\Gamma \rightarrow \mathsf{GL}_{d}(\mathbb{K})$ there exists $C_{\rho}>0$, depending only on $\rho$, such that $$d_{\mathbb{P}}\big( \xi_{\rho}^{1}(x),\xi_{\rho}^{1}(y)\big) \leq C_{\rho}\frac{\sigma_2(\rho(\gamma))}{\sigma_1(\rho(\gamma))}.$$

\noindent \textup{(ii)} for every $(1,1,2)$-hyperconvex representation $\rho:\Gamma \rightarrow \mathsf{GL}_{d}(\mathbb{K})$ there exists $c_{\rho}>0$, depending only on $\rho$, such that $$d_{\mathbb{P}}\big(\xi_{\rho}^{1}(x),\xi_{\rho}^{1}(y)\big) \geq c_{\rho}\frac{\sigma_2(\rho(\gamma))}{\sigma_1(\rho(\gamma))}.$$\end{lemma}

We also need the following bound for the proof of the first part of Theorem \ref{exponent-inverse2}.

\begin{lemma} \label{mainlemma2'} Let $\Gamma$ be a non-elementary word hyperbolic group, $w_0\in \partial_{\infty}X$ and fix $z \in \partial_{\infty}X \smallsetminus \{w_0\}$. There exists $\varepsilon>0$, depending only on $z,w_0\in \partial_{\infty}X$ and $\Gamma$, with the property: if $x,y\in B_{\varepsilon}(w_0)$ are distinct and $(x,y,z)=\gamma(b_1,b_2,b_3)$ for some $(b_1,b_2,b_3)\in \mathcal{F}$, then for every $1$-Anosov representation $\rho:\Gamma \rightarrow \mathsf{GL}_{d}(\mathbb{K})$ we have: $$ d_{\mathbb{P}}\big( \xi_{\rho}^{1}(x),\xi_{\rho}^{1}(y)\big) \geq \omega_{\rho} \frac{\sigma_d(\rho(\gamma))\sigma_{d-1}(\rho(\gamma))}{\sigma_1(\rho(\gamma))^2} $$ where $\omega_{\rho}:=\frac{2}{\pi}\inf \big\{d_{\mathbb{P}}(\xi_{\rho}^1(x_1),\xi_{\rho}^1(x_2)): (x_1,x_2,x_3)\in \mathcal{F} \big\}$.\end{lemma}

\begin{proof} It suffices to prove  for every pair of unit length vectors $v_1,v_2\in \mathbb{K}^d$ the following estimate \begin{equation}\label{lower-est0}d_{\mathbb{P}}\big([gv_1],[gv_2]\big) \geq \frac{2}{\pi} \frac{\sigma_d(g)\sigma_{d-1}(g)}{\sigma_1(g)^2}d_{\mathbb{P}}\big([v_1],[v_2]\big).\end{equation}  We write $$g=k_g \exp(\mu(g))k_{g}' \ \ k_g,k_g' \in \mathsf{K}_d$$ in the Cartan decomposition of $\mathsf{GL}_{d}(\mathbb{K})$ and set $\kappa_{ji}:= \langle k_{g}'v_i,e_j\rangle$ for $1\leq j\leq 2$ and $1 \leq i \leq d.$ By the definition of the metric $d_{\mathbb{P}}$ and the inequality $\frac{2}{\pi}\theta \leq \sin \theta\leq \theta$ for $\theta\in [0,\frac{\pi}{2}]$, we have the double estimate: $$\sqrt{1-|\langle v_1,v_2\rangle|^2}\leq d_{\mathbb{P}}([v_1],[v_2])\leq \frac{\pi}{2}\sqrt{1-|\langle v_1,v_2\rangle|^2}.$$ Then, by using the previous estimate, we obtain the lower bounds \begin{align*} 1-|\langle v_1,v_2\rangle|^2&=1-|\langle k_g'v_1,k_g'v_2\rangle|^2=\sum_{1\leq i<j\leq d} \big|\kappa_{1i}\kappa_{2j}-\kappa_{1j}\kappa_{2i}\big|^2,\end{align*} \begin{align*} d_{\mathbb{P}}\big([gk_1e_1],[gk_2e_1]\big)^2&= d_{\mathbb{P}}\big([\exp(\mu(g))k_g'v_1],[\exp(\mu(g))k_g'v_2]\big)^2\\ &\geq 1-\frac{|\langle \exp(\mu(g))k_g'v_1,\exp(\mu(g))k_g'v_2\rangle |^2}{| |\exp(\mu(g))k_g'v_1 | |^2 | |\exp(\mu(g))k_g'v_2 | |^2}\\ &= \frac{\sum_{1\leq i<j\leq d} \sigma_i(g)^2 \sigma_j(g)^2 \big|\kappa_{1i}\kappa_{2j}-\kappa_{1j}\kappa_{2i}\big|^2}{\big(\sum_{i=1}^{d} \sigma_i(g)^2 |\kappa_{1i}|^2\big) \big(\sum_{i=1}^{d} \sigma_i(g)^2 |\kappa_{2i}|^2 \big)}\\ &\geq \frac{\sigma_{d-1}(g)^2\sigma_d(g)^2}{\sigma_1(g)^4}\big(1-|\langle v_1,v_2\rangle|^2\big)\\ & \geq \frac{4}{\pi^2}\frac{\sigma_{d-1}(g)^2\sigma_d(g)^2}{\sigma_1(g)^4}d_{\mathbb{P}}([v_1], [v_2])^2.  \end{align*}
This completes the proof of (\ref{lower-est0}) and the lemma follows. \end{proof}

\medskip 
\section{Optimal exponential decay constant of the ratio $\frac{\sigma_2}{\sigma_1}$}\label{pf-semisimple-bd}

The main result of this section is the proof of Theorem \ref{semisimple-bound} which we recall immediately.
\begin{theorem}\textup{(Theorem \ref{semisimple-bound})} Let $\rho:\Gamma \rightarrow \mathsf{GL}_d(\mathbb{K})$ be a $1$-Anosov representation. There exists $C>1$ and $m\in \mathbb{Z}$, $0\leq m \leq \max\big\{0,\frac{d(d-1)}{2}-2\big \}$, with the property that for every $\gamma \in \Gamma$ we have: $$C^{-1}e^{-\beta_{\rho}|\gamma|_X}\leq \frac{\sigma_2(\rho(\gamma))}{\sigma_1(\rho(\gamma))} \leq Ce^{-\alpha_{\rho} |\gamma|_X} |\gamma|_X^{m}$$ where $\alpha_{\rho}:= \underset{\gamma \in \Gamma_{\infty}}{\inf} \frac{\log \frac{\ell_1}{\ell_2} (\rho(\gamma))}{|\gamma|_{X,\infty}}$ and $\beta_{\rho}:= \underset{\gamma \in \Gamma_{\infty}}{\sup}\frac{\log \frac{\ell_1}{\ell_2} (\rho(\gamma))}{|\gamma|_{X,\infty}}$.\end{theorem}

We shall make a specific choice of a generating subset of $\Gamma$. Recall that $\Gamma$ acts properly discontinuously and cocompactly on the geodesic space $(X,d_X)$ and choose $L>0$ with the property that for every $x\in X$ we have $\textup{dist}_{X}(\Gamma x_0,x)\leq L$. Let us consider the finite subset of $\Gamma$ \begin{equation}\label{gen-set}\mathcal{R}:=\big \{\gamma \in \Gamma: |\gamma|_X \leq 2L+2 \big \}.\end{equation}

We will  need the following observation.

\begin{observation}\label{gen-set-0} There exists $M>0$, depending only on the model space $(X,d_X)$, with the property: for every $\gamma \in \Gamma$ we can write $\gamma=h_1\cdots h_{p}$, where $h_1,\ldots,h_p\in \mathcal{R}$, $|p-|\gamma|_X|\leq 1$, and for every $1\leq i \leq p$ we have \begin{align*}\big|i-|h_1\cdots h_i|_X\big|&\leq M,\\ \big||h_1\cdots h_i|_X+|(h_1\cdots h_i)^{-1}\gamma|_X-|\gamma|_X\big|&\leq M.\end{align*} \end{observation}

\begin{proof} Let $\gamma \in \Gamma$ such that $d_X(\gamma x_0,x_0)>1$. Consider a geodesic $[x_0,\gamma x_0]\subset X$ and points $x_1,\ldots,x_p=\gamma x_0$ in $[x_0,\gamma x_0]$ such that $d_X(x_i,x_{i+1})=1$ for $0\leq i \leq p-1$ and $d_X(x_p,\gamma x_0)\leq 1$. For every $i$, choose $g_i\in \mathcal{R}$ such that $d_X(g_ix_0,x_i)\leq L$ and observe that $h_{i}:=g_i^{-1}g_{i+1}\in \mathcal{R}$ since $$d_X(g_{i}^{-1}g_{i+1}x_0,x_0)\leq d_X(x_i,x_{i+1})+d_X(g_ix_0,x_i)+d_X(g_{i+1}x_0,x_{i+1})\leq 2L+1.$$ In particular, we can write $g=h_1\cdots h_p$ such that $\big||h_1\cdots h_i|_X-i\big |=\big|i-d_{X}(g_ix_0,x_0)\big|\leq L+1$ for every $1\leq i \leq p$. Moreover, note that \begin{align*}\big| |(h_1\cdots h_i)^{-1}\gamma|_X-(|\gamma|_X-|h_1\cdots h_i|_X)\big| &= \big|d_{X}(\gamma x_0,g_ix_0)-d_{X}(\gamma x_0,x_0)+d_{X}(g_ix_0,x_0)\big|\\ & \leq  \big|d_{X}(\gamma x_0,x_i)-d_{X}(\gamma x_0,x_0)+d_{X}(x_i,x_0)\big|+L=L, \end{align*} and the observartion follows.\end{proof}

\begin{lemma}\label{gen-set-1} Let $\rho:\Gamma \rightarrow \mathsf{GL}_d(\mathbb{K})$ be a $1$-Anosov representation. There exists $\varepsilon>0$ with the property: if $\gamma \in \Gamma$ and we write $\gamma=h_1\cdots h_r$, $h_1,\ldots,h_r\in \mathcal{R}$, as in Observation \ref{gen-set-0}, then $$\min_{1\leq i \leq r-1}\frac{\sigma_1(\rho(h_1\cdots h_r))}{\sigma_1(\rho(h_1\cdots h_i))\sigma_1(\rho(h_{i+1}\cdots h_r))} \geq \varepsilon.$$\end{lemma}

\begin{proof} Note that for every $1\leq i \leq r$, by the choice of $h_1,\ldots, h_r\in \mathcal{R}$, we have that $$|h_1\cdots h_i|_X+|h_{i+1}\cdots h_{r}|_X-|h_1\cdots h_r|_X\leq 2M+1$$ where $M>0$ is furnished by Observation \ref{gen-set-0}. Since $\rho$ is $1$-Anosov, by \cite[Prop. 1.12]{Ts20}, there exists $0<\delta<1$ and $t>0$, depending only on $\rho$, such that for every $w_1,w_2 \in \Gamma$ we have: \begin{align}\label{prod-ineq-0}\frac{\sigma_1(\rho(w_1 w_2))}{\sigma_1(\rho(w_1))\sigma_1(\rho(w_2))}&\geq  \delta \exp \big(t(|w_1 w_2|_X-|w_1|_X-|w_2|_X)\big).\end{align} The conclusion now follows by considering $\varepsilon:=\delta e^{-t(2M+1)}$ and applying (\ref{prod-ineq-0}) for $w_1:=h_1\cdots h_i$ and $w_2:=h_{i+1}\cdots h_r$.  \end{proof}

\begin{proof}[Proof of Theorem \ref{semisimple-bound} \textup{(}upper bound\textup{)}] Let $\rho^{ss}:\Gamma \rightarrow \mathsf{GL}_d(\mathbb{K})$ be a semisimplification of $\rho$. Note that since $\rho$ and $\rho^{ss}$ have the same Lyapunov projection (see \cite[Prop. 1.8]{GGKW}), $\rho^{ss}$ is also $1$-Anosov. In particular $\alpha_{\rho}=\alpha_{\rho^{ss}}$. Now let $F\subset \Gamma$ be a finite subset and $C_1>0$ satisfying the conclusion of Theorem \ref{AMS-KP}. For every $\gamma \in \Gamma$ there exists $f\in F$: $$\frac{\sigma_2(\rho^{ss}(\gamma))}{\sigma_1(\rho^{ss}(\gamma))} \leq e^{2C_1}\frac{\ell_2(\rho(\gamma f))}{\ell_1(\rho(\gamma f))}\leq e^{2C_1}e^{-\alpha_{\rho}|\gamma f|_{X, \infty}}\leq e^{(2+\alpha_{\rho})C_1}e^{-\alpha_{\rho} |\gamma|_X}.$$ Therefore, for $C_2:=e^{2+\alpha_{\rho}}C_1$, we conclude that \begin{equation}\label{12-semisimple} \frac{\sigma_2(\rho^{ss}(\gamma))}{\sigma_1(\rho^{ss}(\gamma))} \leq C_2e^{-\alpha_{\rho} |\gamma|_X} \ \ \forall \ \gamma \in \Gamma.\end{equation} 

 Let $(\wedge^2 \rho)^{ss}:\Gamma \rightarrow \mathsf{GL}(\wedge^2\mathbb{K}^d)$ be a semisimplification of the exterior power $\wedge^2 \rho$. Up to conjugation by an element of $\mathsf{GL}(\wedge^2 \mathbb{K}^d)$, we may assume that there exists a decomposition $\wedge^2 \mathbb{K}^d=V_1 \oplus \cdot \cdot \cdot \oplus V_{\ell}$, $1\leq \ell\leq \frac{d(d-1)}{2}-1$, such that $$\wedge^2 \rho=\begin{pmatrix}
\rho_1 & \cdots & \ast \\ 
  & \ddots  & \vdots \\ 
 &  & \rho_{\ell}
\end{pmatrix},\  (\wedge^2 \rho)^{ss}=\begin{pmatrix}
\rho_1 &   & \\ 
  & \ddots  &  \\ 
 &  & \rho_{\ell}
\end{pmatrix},$$ where $\{\rho_i:\Gamma \rightarrow \mathsf{GL}(V_i)\}_{i=1}^{\ell}$ are irreducible. Note that since $$\ell_1\big((\wedge^2 \rho)^{ss}(h)\big)=\ell_1(\rho^{ss}(h))\ell_2(\rho^{ss}(h)), \ \ \forall \ h\in \Gamma$$ by Theorem \ref{finitesubset} there exists $C_3>1$ such that \begin{equation} \label{double-bound-semisimple} C_3^{-1}\leq \frac{\sigma_1((\wedge^2\rho)^{ss}(h))}{\sigma_1(\rho^{ss}(h)) \sigma_2(\rho^{ss}(h))}\leq C_3 \ \ \forall \ h\in \Gamma.\end{equation} 

For $1\leq i \leq \ell$ and $\gamma \in \Gamma$, let $u_i(\gamma)$ be the $(i+1)$-th column block in the matrix decomposition of $\wedge^2 \rho$ excluding the block $\rho_{i+1}(\gamma)$. We define the representations $\psi_i:\Gamma \rightarrow \mathsf{GL}_{m_i}(\mathbb{K})$, $m_i=\sum_{j=1}^{i}\textup{dim}V_{j}$, $i=1,\ldots, \ell-1$, recursively as follows: \begin{equation}\label{decompose-1} \psi_1(\gamma)=\rho_1(\gamma), \ \psi_{i+1}(\gamma)=\begin{pmatrix}[0.8]
\psi_i(\gamma) & u_{i}(\gamma)  \\ 
  & \rho_{i+1}(\gamma)  \\ 
\end{pmatrix},\ \psi_{\ell}(\gamma)=\wedge^2\rho(\gamma)  \ \ \gamma \in \Gamma.\end{equation} 

Observe that  (\ref{12-semisimple}) and (\ref{double-bound-semisimple}) imply that \begin{align}\label{irr-ineq}\max_{1\leq i\leq \ell}\sigma_1(\rho_{i}(\gamma))\leq \sigma_1\big((\wedge^2\rho)^{ss}(\gamma)\big)\leq C_2C_3 \sigma_1(\rho^{ss}(\gamma))^2e^{-\alpha_{\rho}|\gamma|_X} \ \ \forall \ \gamma \in \Gamma.\end{align}

By using induction, we will prove that for every $1\leq q\leq \ell$, there exists $D_q>1$: \begin{equation}\label{ineq-semi-1}\sigma_1(\psi_q(\gamma))\leq D_q \sigma_1(\rho^{ss}(\gamma))^2e^{-a_{\rho}|\gamma|_X}|\gamma|_{X}^{q-1} \ \ \forall \ \gamma \in \Gamma.\end{equation} Recall the definition of the finite generating set $\mathcal{R}\subset \Gamma$ from (\ref{gen-set}) and set $$D_0:=\max_{h\in \mathcal{R}}\big(||\wedge^2\rho(h)||\cdot ||\wedge^2 \rho(h^{-1})||\big)\geq 1.$$ 

Let $\gamma \in \Gamma$ be an arbitrary element with $|\gamma|_X>1$ and write $\gamma=h_1\cdots h_{p}$ as in Observation \ref{gen-set-0} for some $p\in \mathbb{N}$ with $|p-|\gamma|_X|\leq 1$. Let us also set $$\mathsf{U}_{0\gamma}:=e,\ \mathsf{U}_{j\gamma}:=h_1\cdots h_j \ \ j=1,\ldots, p$$ and note that there is $C_4>0$, depending only on the model space $(X,d_X)$, such that \begin{equation} \label{gen-set-eq3}\max_{1\leq j \leq p}\big(\big||\mathsf{U}_{j\gamma}|_X+|\mathsf{U}_{j\gamma}^{-1}\gamma |_X-|\gamma|_X \big|\big)\leq C_4.\end{equation} 

We first note that for $q=1$, (\ref{ineq-semi-1}) follows immediately by (\ref{irr-ineq}). Now suppose that (\ref{ineq-semi-1}) holds true for $\psi_q$, $1\leq q\leq \ell-1$, and we will prove it for $\psi_{q+1}$. Observe that we have the equality \begin{equation} \label{gen-set-eq2} u_q(\gamma)=u_q(h_1 \cdot \cdot \cdot h_{p})=\sum_{j=0}^{p-1} \psi_q(h_{0} \cdot \cdot \cdot h_{j}) u_q(h_{j+1}) \rho_{q+1}(h_{j+2}\cdot \cdot \cdot h_{p+1})\end{equation} where $h_0=h_{p+1}:=e$. Since $\rho^{ss}$ is $1$-Anosov, by Lemma \ref{gen-set-1} there exists $\varepsilon>0$, depending only on $\rho$, such that \begin{equation} \label{gen-set-eq1}\min_{1\leq j\leq p}\frac{\sigma_1(\rho^{ss}(\gamma))}{\sigma_1(\rho^{ss}(\mathsf{U}_{j\gamma}))\sigma_1(\rho^{ss}(\mathsf{U}_{j\gamma}^{-1}\gamma))}\geq \varepsilon.\end{equation}

Then, by using (\ref{gen-set-eq3}), (\ref{gen-set-eq2}) and (\ref{gen-set-eq1}), we successively obtain the following bounds: \begin{align*}\frac{\big|\big|u_q(\gamma)\big|\big|}{\sigma_1(\rho^{ss}(\gamma))^2} & \leqslant D_0^2 \sum_{j=0}^{p-1} \frac{1}{\sigma_1(\rho^{ss}(\gamma))^2} \sigma_1\big(\psi_{q}(\mathsf{U}_{j\gamma})\big) \sigma_1\big(\rho_{q+1}(\mathsf{U}_{j\gamma}^{-1}\gamma)\big)\\ \end{align*} \begin{align*}  \ \ \ \ \ \ \ \ \ \ \ \ \ \ \ \ \ \ \   & \leqslant D_0^2 D_{q}\sum_{j=0}^{p-1} \frac{|\mathsf{U}_{j\gamma}|_X^{q-1}}{\sigma_1(\rho^{ss}(\gamma))^2} e^{-\alpha_{\rho}|\mathsf{U}_{j\gamma}|_X} \sigma_1\big(\rho^{ss}(\mathsf{U}_{j\gamma})\big)^2  \sigma_1\big((\wedge^2 \rho)^{ss}(\mathsf{U}_{j\gamma}^{-1}\gamma)\big)\\ &\leq D_0^2 D_qC_3 \sum_{j=0}^{p-1}|\mathsf{U}_{j\gamma}|_{X}^{q-1} e^{-\alpha_{\rho}|\mathsf{U}_{j\gamma}|_X} \frac{\sigma_1(\rho^{ss}(\mathsf{U}_{j\gamma}))^2  \sigma_1(\rho^{ss}(\mathsf{U}_{j\gamma}^{-1}\gamma))^2}{\sigma_1(\rho^{ss}(\gamma))^2}\frac{\sigma_2(\rho^{ss}(\mathsf{U}_{j\gamma}^{-1}\gamma))}{\sigma_1(\rho^{ss}(\mathsf{U}_{j\gamma}^{-1}\gamma))}\\ & \leq D_0^2 D_qC_3\varepsilon^{-2} \sum_{j=0}^{p-1}|\mathsf{U}_{j\gamma}|_{X}^{q-1} e^{-\alpha_{\rho}|\mathsf{U}_{j\gamma}|_X} \frac{\sigma_2(\rho^{ss}(\mathsf{U}_{j\gamma}^{-1}\gamma))}{\sigma_1(\rho^{ss}(\mathsf{U}_{j\gamma}^{-1}\gamma))}\\ & \leq D_0^2 D_q C_2C_3\varepsilon^{-2} \sum_{j=0}^{p-1}|\mathsf{U}_{j\gamma}|_{X}^{q-1} e^{-\alpha_{\rho}(|\mathsf{U}_{j\gamma}|_X+|\mathsf{U}_{j\gamma}^{-1}\gamma|_X)}\\ &\leq D_0^2 D_q C_2C_3\varepsilon^{-2} \sum_{j=0}^{p-1}|\mathsf{U}_{j\gamma}|_{X}^{q-1} e^{-\alpha_{\rho}(|\gamma|_X-C_4)} \\  &\leq D_0^2 D_q C_2 e^{C_4 \alpha_{\rho}}\varepsilon^{-2} p^{q} e^{-\alpha_{\rho}|\gamma|_X} \\  &\leq D_0^2 D_q C_2 C_3e^{C_4 \alpha_{\rho}}\varepsilon^{-2} 2^{q}|\gamma|_X^{q} e^{-\alpha_{\rho}|\gamma|_X}. \end{align*}

As a consequence of the the previous estimate and  (\ref{irr-ineq}), we obtain a constant $D_{q+1}>0$, dependng only on $\rho$, such that for every $\gamma \in \Gamma$, \begin{align*}\sigma_1(\psi_{q+1}(\gamma))&\leq \sigma_1(\rho_{q+1}(\gamma))+\sigma_1(\psi_{q}(\gamma))+\big|\big|u_q(\gamma)\big|\big|\\ &\leq  D_{q+1}\sigma_1(\rho^{ss}(\gamma))^2e^{-\alpha_{\rho}|\gamma|_X} |\gamma|_X^{q}.\end{align*} The induction is complete and (\ref{ineq-semi-1}) follows.

Now we finish the proof of the upper bound. \hbox{Since $\rho$ is $1$-Anosov, we may choose $C_5>1$ such that} \begin{align}\label{ratio-semisimple}C_5^{-1}\leq \frac{\sigma_1(\rho(\gamma))}{\sigma_1(\rho^{ss}(\gamma))}\leq C_5 \ \ \forall \ \gamma \in \Gamma.\end{align} In particular, since $\psi_{\ell}=\wedge^2\rho$ and $1\leq \ell\leq \frac{d(d-1)}{2}-1$, we conclude for every $\gamma \in \Gamma$:\begin{align*}\frac{\sigma_2(\rho(\gamma))}{\sigma_1(\rho(\gamma))}=\frac{\sigma_1(\psi_{\ell}(\gamma))}{\sigma_1(\rho(\gamma))^2}&\leq \frac{\sigma_1(\rho^{ss}(\gamma))^2}{\sigma_1(\rho(\gamma))^2} D_{\ell}e^{-\alpha_{\rho}|\gamma|_X}|\gamma|_X^{\ell-1}\\ &\leq C_5^2D_{\ell}e^{-\alpha_{\rho}|\gamma|_X}|\gamma|_X^{\ell-1}.\end{align*} This last estimate concludes the proof of the upper bound.\end{proof}

\begin{proof}[Proof of Theorem \ref{semisimple-bound} \textup{(}lower bound\textup{)}] Let us recall that $\beta_{\rho}:=\sup_{\gamma \in \Gamma_{\infty}}\frac{\log\frac{\ell_1}{\ell_2}(\rho(\gamma))}{|\gamma|_{X,\infty}}$. By applying Theorem \ref{AMS-KP} for the semisimplification $\rho^{ss}$ of $\rho$ and the definition of $\beta_{\rho}>0$, we obtain $C_6>1$ with the property: $$\frac{\sigma_2(\rho^{ss}(\gamma))}{\sigma_1(\rho^{ss}(\gamma))}\geq C_6^{-1}e^{-\beta_{\rho}|\gamma|_X} \ \ \forall \gamma\in \Gamma.$$ Therefore, by using (\ref{double-bound-semisimple}) and (\ref{ratio-semisimple}), for every $\gamma \in \Gamma$ we conclude that \begin{align*}\frac{\sigma_2(\rho(\gamma))}{\sigma_1(\rho(\gamma))}=\frac{\sigma_1(\wedge^2 \rho(\gamma))}{\sigma_1(\rho(\gamma))^2}&\geq C_{5}^{-2}\frac{\sigma_1((\wedge^2\rho)^{ss}(\gamma))}{\sigma_1(\rho^{ss}(\gamma))^2}\\ &\geq (C_5^2C_3)^{-1}\frac{\sigma_2(\rho^{ss}(\gamma))}{\sigma_1(\rho^{ss}(\gamma))}\geq (C_6 C_5^2C_3)^{-1}e^{-\beta_{\rho}|\gamma|_X}.\end{align*} This concludes the proof of the lower bound. \end{proof} 

As a corollary of Theorem \ref{semisimple-bound}, we obtain the following relations involving the Lyapunov and Cartan projection of an $1$-Anosov representation. 

\begin{corollary}\label{equality-exp} Let $\rho:\Gamma \rightarrow \mathsf{GL}_d(\mathbb{K})$ be a $1$-Anosov representation. \hbox{Then the following equalities hold:} \begin{align*} \inf_{\gamma\in \Gamma_{\infty}}\frac{\log \frac{\ell_1}{\ell_2}(\rho(\gamma))}{|\gamma|_{X,\infty}}&=\sup_{n\geq 1}\inf_{|\gamma|_X\geq n}\frac{\log \frac{\sigma_1}{\sigma_2}(\rho(\gamma))}{|\gamma|_{X}}\\  \sup_{\gamma\in \Gamma_{\infty}}\frac{\log \frac{\ell_1}{\ell_2}(\rho(\gamma))}{|\gamma|_{X,\infty}}&=\inf_{n\geq 1}\sup_{|\gamma|_X\geq n}\frac{\log \frac{\sigma_1}{\sigma_2}(\rho(\gamma))}{|\gamma|_{X}}.\end{align*}\end{corollary}

We also establish the following bound comparing the ratio of the first and second singular value between two $1$-Anosov representations.

\begin{theorem} Let $\rho:\Gamma \rightarrow \mathsf{GL}_d(\mathbb{K})$ and $\psi:\Gamma \rightarrow \mathsf{GL}_m(\mathbb{K})$ be two representations. Suppose that $\rho$ is $1$-Anosov and $\psi$ is $\{1,2\}$-Anosov. There exists $J>0$ and $m\in \mathbb{Z}$, $0\leq m \leq \max\big\{0,\frac{d(d-1)}{2}-2\big\}$, such that for every $\gamma \in \Gamma$ we have: $$\frac{\sigma_2(\rho(\gamma))}{\sigma_1(\rho(\gamma))}\leq J \frac{\sigma_2(\psi(\gamma))^{\alpha_{\psi,\rho}}}{\sigma_1(\psi(\gamma))^{\alpha_{\psi,\rho}}} |\gamma|_X^{m}$$ where $\alpha_{\psi,\rho}:=\underset{\gamma \in \Gamma_{\infty}}{\inf}\frac{\log \frac{\ell_1}{\ell_2}(\rho(\gamma))}{\log\frac{\ell_1}{\ell_2}(\psi(\gamma))}.$\end{theorem}

\begin{proof} Let $\rho^{ss}:\Gamma \rightarrow \mathsf{GL}_d(\mathbb{K})$ be a semisimplification of $\rho$. Let $R>0$ and $F\subset \Gamma$ be a finite subset satisfying the conclusion of Proposition \ref{12finitesubset} for the semisimple representation $\rho^{ss}$ and the $\{1,2\}$-Anosov representation $\psi$. For $\gamma \in \Gamma$, there exists $f\in F$ such that: \begin{align*} \frac{\sigma_2(\rho^{ss}(\gamma))}{\sigma_1(\rho^{ss}(\gamma))}&\leq e^{2R} \frac{\ell_2(\rho^{ss}(\gamma f))}{\ell_1(\rho^{ss}(\gamma f))}\leq e^{2R}  \frac{\ell_2(\psi(\gamma f))^{\alpha_{\psi,\rho}}}{\ell_1(\psi(\gamma f))^{\alpha_{\psi,\rho}}}\\ &\leq e^{2R(1+\alpha_{\psi,\rho})} \frac{\sigma_2(\psi(\gamma))^{\alpha_{\psi,\rho}}}{\sigma_1(\psi(\gamma))^{\alpha_{\psi,\rho}}}.\end{align*} In particular, if we set $L_1:=e^{2R(1+\alpha_{\psi,\rho})}$, we conclude that: \begin{align}\label{ineq-semi-12} \frac{\sigma_2(\rho^{ss}(\gamma))}{\sigma_1(\rho^{ss}(\gamma))}\leq L_1 \frac{\sigma_2(\psi(\gamma))^{\alpha_{\psi,\rho}}}{\sigma_1(\psi(\gamma))^{\alpha_{\psi,\rho}}}, \ \
\forall \gamma \in \Gamma.\end{align} Therefore, the theorem holds true when $\rho$ is semisimple. 

Now we work similarly as in the proof of the upper bound of Theorem \ref{semisimple-bound}. Let $\gamma \in \Gamma$ with \hbox{$|\gamma|_X>1$,} write $\gamma=h_1\cdots h_p$, $h_i\in \mathcal{R}$ as in Observation \ref{gen-set-0}, and recall that we set $$\mathsf{U}_{0\gamma}:=e,\ \mathsf{U}_{j\gamma}:=h_1\cdots h_j \ \ j=1,\ldots,p.$$ By applying Lemma \ref{gen-set-1} for the $1$-Anosov representations $\psi$ and $\wedge^2\psi$, there exists $\epsilon>0$, depending only on $\psi$, such that: \begin{align}\label{ineq-wedge} \min_{1\leq j\leq \ell}\frac{\sigma_1(\psi (\gamma))}{\sigma_1(\psi(\mathsf{U}_{j\gamma}))\sigma_1(\psi(\mathsf{U}_{j\gamma}^{-1}\gamma))}\geq \epsilon,\ \min_{1\leq j\leq \ell} \frac{\sigma_1(\wedge^2 \psi (\gamma))}{\sigma_1(\wedge^2 \psi(\mathsf{U}_{j\gamma}))\sigma_1(\wedge^2 \psi(\mathsf{U}_{j\gamma}^{-1}\gamma))}\geq \epsilon. \end{align}

Now, without loss of generality, we may consider decompositions $\wedge^2\mathbb{K}^d=V_1\oplus\cdots \oplus V_{\ell}$ for $(\wedge^2 \rho)^{ss}$ with blocks $\psi_i:\Gamma\rightarrow \mathsf{GL}(V_i)$, $1\leq \ell\leq \frac{d(d-1)}{2}-1$, similarly as in (\ref{decompose-1}). We are going to prove inductively that for every $1\leq q\leq \ell$, there exists $R_{q}>0$ such that: \begin{equation}\label{ineq-semi-123}\sigma_1(\psi_q(\gamma))\leq R_q \sigma_1(\rho^{ss}(\gamma))^2  \frac{\sigma_2(\psi(\gamma))^{\alpha_{\psi,\rho}}}{\sigma_1(\psi(\gamma))^{\alpha_{\psi,\rho}}}|\gamma|_{X}^{q-1} \ \ \forall \gamma \in \Gamma.\end{equation} Note that the statement is true for $q=1$, thanks to the fact that $\psi_0:=\rho_1$ is semisimple. Now suppose the statement holds true for $q<\ell$. By working similarly as in the proof of Theorem \ref{semisimple-bound}, using the inductive step and (\ref{double-bound-semisimple}), (\ref{gen-set-eq2}), (\ref{ineq-semi-12}), (\ref{ineq-wedge}), we obtain the estimates: \begin{align*}\frac{| |u_q(\gamma) | |}{\sigma_1(\rho^{ss}(\gamma))^2} & \leqslant D_0^2 \sum_{j=0}^{p-1} \frac{1}{\sigma_1(\rho^{ss}(\gamma))^2} \sigma_1\big(\psi_{q}(\mathsf{U}_{j\gamma})\big) \sigma_1\big(\rho_{q+1}(\mathsf{U}_{j\gamma}^{-1}\gamma)\big)\\ &\leq D_0^2 R_q \sum_{j=0}^{p-1} \big|\mathsf{U}_{j\gamma}\big|_X^q \frac{\sigma_2(\psi(\mathsf{U}_{j\gamma}))^{\alpha_{\psi,\rho}}}{\sigma_1(\psi(\mathsf{U}_{j\gamma}))^{\alpha_{\psi,\rho}}}\frac{\sigma_1(\rho^{ss}(\mathsf{U}_{j\gamma}))^2}{\sigma_1(\rho^{ss}(\gamma))^2}\sigma_1\big((\wedge^2\rho)^{ss}(\mathsf{U}_{j\gamma}^{-1}\gamma)\big)\\ &\leq D_0^2 R_qC_3\epsilon^{-2} \sum_{j=0}^{p-1} \big|\mathsf{U}_{j\gamma}\big|_X^{q-1} \frac{\sigma_2(\psi(\mathsf{U}_{j\gamma}))^{\alpha_{\psi,\rho}}}{\sigma_1(\psi(\mathsf{U}_{j\gamma}))^{\alpha_{\psi,\rho}}}\frac{\sigma_2(\rho^{ss}(\mathsf{U}_{j\gamma}^{-1}\gamma))}{\sigma_1(\rho^{ss}(\mathsf{U}_{j\gamma}^{-1}\gamma))}\\ &\leq D_0^2 R_qC_3L_1\epsilon^{-2}\sum_{j=0}^{p-1} \big|\mathsf{U}_{j\gamma}\big|_X^{q-1}  \frac{\sigma_2(\psi(\mathsf{U}_{j\gamma}))^{\alpha_{\psi,\rho}}}{\sigma_1(\psi(\mathsf{U}_{j\gamma}))^{\alpha_{\psi,\rho}}}\frac{\sigma_2(\psi(\mathsf{U}_{j\gamma}^{-1}\gamma))^{\alpha_{\psi,\rho}}}{\sigma_1(\psi(\mathsf{U}_{j\gamma}^{-1}\gamma))^{\alpha_{\psi,\rho}}}\\ &= D_0^2 R_qC_3L_1\epsilon^{-2}\sum_{j=0}^{p-1} \big|\mathsf{U}_{j\gamma}\big|_X^{q-1}  \frac{\sigma_1(\wedge^2\psi(\mathsf{U}_{j\gamma}))^{\alpha_{\psi,\rho}}}{\sigma_1(\psi(\mathsf{U}_{j\gamma}))^{2\alpha_{\psi,\rho}}}\frac{\sigma_1(\wedge^2 \psi(\mathsf{U}_{j\gamma}^{-1}\gamma))^{\alpha_{\psi,\rho}}}{\sigma_1(\psi(\mathsf{U}_{j\gamma}^{-1}\gamma))^{2\alpha_{\psi,\rho}}}\\ & \leq D_0^2 R_qC_3L_1\epsilon^{-2-3\alpha_{\psi,\rho}} \sum_{j=0}^{p-1} \big|\mathsf{U}_{j\gamma}\big|_X^{q-1}  \frac{\sigma_2(\psi(\gamma))^{\alpha_{\psi,\rho}}}{\sigma_1(\psi(\gamma))^{\alpha_{\psi,\rho}}}\\ &  \leq 2^{q} D_0^2 R_qC_3L_1\epsilon^{-2-3\alpha_{\psi,\rho}}  \frac{\sigma_2(\psi(\gamma))^{\alpha_{\psi,\rho}}}{\sigma_1(\psi(\gamma))^{\alpha_{\psi,\rho}}} |\gamma |_X^{q}. \end{align*} Using this last estimate we deduce that $$\sigma_1(\psi_{q+1}(\gamma))\leq \sigma_1(\rho_{q+1}(\gamma))+\sigma_1(\psi_q(\gamma))+\big|\big|u_q(\gamma)\big|\big|\leq R_{q+1} \frac{\sigma_2(\psi(\gamma))^{\alpha_{\psi,\rho}}}{\sigma_1(\psi(\gamma))^{\alpha_{\psi,\rho}}} |\gamma |_X^{q} \ \ \forall \gamma \in \Gamma,$$ where $R_{q+1}>0$ is a constant depending only on $\psi$ and $\rho$. This completes the proof of the induction and of (\ref{ineq-semi-123}). Since $\psi_{\ell}=\wedge^2 \rho$, by applying (\ref{ineq-semi-123}) for $q=\ell$ we  finish the proof of the estimate. \end{proof}

\begin{corollary}\label{equality-exp-1}  Let $\rho:\Gamma \rightarrow \mathsf{GL}_d(\mathbb{K})$ and $\psi:\Gamma \rightarrow \mathsf{GL}_m(\mathbb{K})$ be two representations. Suppose that $\rho$ is $1$-Anosov and $\psi$ is $\{1,2\}$-Anosov. \hbox{Then the following equality holds:} \begin{align*} \inf_{\gamma\in \Gamma_{\infty}}\frac{\log \frac{\ell_1}{\ell_2}(\rho(\gamma))}{\log \frac{\ell_1}{\ell_2}(\psi(\gamma))}&=\sup_{n\geq 1}\inf_{|\gamma|_X\geq n}\frac{\log \frac{\sigma_1}{\sigma_2}(\rho(\gamma))}{\log \frac{\sigma_1}{\sigma_2}(\psi(\gamma))}.\end{align*}\end{corollary}

\section{Proof of Theorem \ref{exponent}} \label{proofthm1.2}

In this section we prove Theorem \ref{exponent}. Recall that we fix a model space $(X,d_X)$ for the hyperbolic group $\Gamma$ and a visual metric $d_a,a>1,$ on $\partial_{\infty}X$ satisfying (\ref{visualineq}). Recall that for a representaion $\rho:\Gamma \rightarrow \mathsf{GL}_d(\mathbb{K})$, a $\rho$-equivariant map $\xi_{\rho}:\partial_{\infty}X\rightarrow \mathbb{P}(\mathbb{K}^d)$ is called spanning if $\langle \xi_{\rho}(\partial_{\infty}X)\rangle=\mathbb{K}^d$.

We will need the following lemma which generalizes \cite[Lem. 6.8]{Sambarino} and gives an upper bound for the H\"older exponent of the Anosov limit maps in terms of singular value gaps.

\begin{lemma} \label{doublebound} Let $\rho:\Gamma \rightarrow \mathsf{GL}_{d}(\mathbb{K})$ be a representation which admits a continuous $\rho$-equivariant spanning map $\xi_{\rho}: (\partial_{\infty}X,d_a)\rightarrow (\mathbb{P}(\mathbb{K}^d),d_{\mathbb{P}})$. Suppose that $\xi_{\rho}$ is bi-H\"older continuous, i.e. there exist $\beta\geq \alpha>0$ and $C>1$ such that for every $x,y \in \partial_{\infty}X$, $$C^{-1}d_a(x,y)^{\beta}\leq d_{\mathbb{P}}\big(\xi_{\rho}(x),\xi_{\rho}(y)\big) \leq Cd_{a}(x,y)^{\alpha}.$$Let $(\gamma_n)_{n \in \mathbb{N}}$ be an infinite sequence of elements of $\Gamma$. Then \begin{align*} \alpha \leq \frac{1}{\log a} \varliminf_{n \rightarrow \infty}\frac{\log \frac{\sigma_1}{\sigma_2}(\rho(\gamma_n))}{|\gamma_n|_{X}} \leq  \frac{1}{\log a} \varlimsup_{n \rightarrow \infty}\frac{\log \frac{\sigma_1}{\sigma_2}(\rho(\gamma_n))}{|\gamma_n|_{X}}\leq \beta.\end{align*} \end{lemma}

\begin{proof} Without loss of generality we may assume that $\lim_n\frac{\log \frac{\sigma_1}{\sigma_2}(\rho(\gamma_n))}{|\gamma_n|_{X}}$ exists, so we continue working under this assumption. We may write $$\rho(\gamma_n)=k_{n}\exp(\mu(\rho(\gamma_n)))k_{n}'\ k_{n},k_{n}'\in \mathsf{K}_d,$$ in the Cartan decomposition of $\mathsf{GL}_d(\mathbb{K})$ and, up to passing to a subsequence, we may assume that $\lim_{n}k_{n}'=k'$.

Since $\xi_{\rho}$ is spanning and $\Gamma$ acts minimally on $\partial_{\infty}X$, for every open subset $B\subset \partial_{\infty}X$, $\xi_{\rho}(B)$ spans $\mathbb{K}^d$ and is not contained in a finite union of projective hyperplanes\footnote{To see this, assume that $\xi_{\rho}(B)\subset \bigcup_{i=1}^{s}\mathbb{P}(V_i)$, with $s$ minimal and $\textup{dim}_{\mathbb{K}}V_i=d-1$. Then either $s=1$ or $s\geq 2$ and there is $x\in B$ and $i\in \{1,\ldots,s\}$ such that $x\in \mathbb{P}(V_i)\smallsetminus \bigcup_{j\neq i}\mathbb{P}(V_j)$. In the first case, $\xi_{\rho}(B)\subset \mathbb{P}(V_i)$ and in the second, by the choice of $x$, there is $B'\subset B$ open, containing $x$, such that $\xi_{\rho}(B')\subset \mathbb{P}(V_i)$. However, this is absurd since $\xi_{\rho}(B')$ spans $\mathbb{K}^d$.}. We may fix $w_0\in \partial_{\infty}X$ such that $\xi_{\rho}(w_0)\notin \mathbb{P}((k')^{-1}\langle e_2,\ldots,e_d\rangle)\cup \mathbb{P}((k')^{-1}\langle \{e_r:r\neq 2\}\rangle)$ and choose $U\subset \mathbb{P}(\mathbb{K}^d)$ an open set with $\xi_{\rho}(w_0)\in U$ and a local trivialization $p:U\rightarrow \mathsf{K}_d$ of the bundle $\pi:\mathsf{K}_d\twoheadrightarrow \mathbb{P}(\mathbb{K}^d)$ (i.e. $y=[p(y)e_1]$ for every $y\in U$). By considering $\varepsilon>0$ small enough we may assume that:
\medskip

\noindent $\textup{(i)}$ if $d_a(w_0,y)<\varepsilon$, $\xi_{\rho}^{1}(y)=[h(y)e_1]$, then $\langle k'h(y)e_1,e_1\rangle \neq 0$ and $\langle k'h(y) e_1,e_2 \rangle \neq 0$.

\noindent $\textup{(ii)}$ the continuous map $g:B_{\varepsilon}(w_0)\rightarrow \mathbb{R}$, $y \mapsto g(y):=\frac{\langle k'h(y)e_1,e_2\rangle }{\langle k'h(y)e_1,e_1 \rangle}$, is not constant.

\medskip
\noindent Therefore, we may choose $z,x \in B_{\varepsilon}(w_0)$ such that $g(z)\neq g(x)$, $z,x \neq \lim_{n}\gamma_n^{-1}x_0$. We set \begin{align*} a_{z,i,n}&:=\frac{\sigma_i(\rho(\gamma_n))}{\sigma_1(\rho(\gamma_n))}\frac{\langle k_{n}'h(z)e_1,e_i \rangle }{\langle k_{n}'h(z)e_1,e_1 \rangle}, \ \ v_{z,n}:=\sum_{i>2}^{}a_{z,i,n} e_i, \\  a_{x,i,n}&:=\frac{\sigma_i(\rho(\gamma_n))}{\sigma_1(\rho(\gamma_n))}\frac{\langle k_{n}'h(x)e_1,e_i \rangle }{\langle k_{n}'h(x)e_1,e_1 \rangle}, \ \ v_{x,n}:=\sum_{i>2}^{}a_{x,i,n} e_i.\end{align*} Since $z,x \in U$, there exists $0<\delta<1$ with $|\langle k_{n}'h(z)e_1,e_1\rangle| \geqslant \delta$ and $|\langle k_{n}'h(x)e_1,e_1\rangle| \geqslant \delta$ for every $n \in \mathbb{N}$. In particular, for every $2 \leq i \leq d$, we have $$\max \big \{|a_{z,i,n}|, |a_{x,i,n}|\big\} \leq \frac{\sigma_2(\rho(\gamma_n))}{\sigma_1(\rho(\gamma_n))}\frac{1}{\delta}, \ \max \big \{ ||v_{z,n}||, ||v_{x,n}||\big\} \leq \frac{\sigma_2(\rho(\gamma_n))}{\sigma_1(\rho(\gamma_n))}\frac{d-2}{\delta}.$$

Now we observe that for every $n\in \mathbb{N}$, \begin{align*}  & d_{\mathbb{P}}\big(\xi_{\rho}(\gamma_n z),\xi_{\rho}(\gamma_n x)\big)^2=d_{\mathbb{P}}\big([\exp(\mu(\rho(\gamma_n))k_{n}'h(z)e_{1}],[\exp(\mu(\rho(\gamma_n))k_{n}'h(x)e_1] \big)^2 \\ &=\frac{(a_{z,2,n}-a_{x,2,n})^2+||a_{x,2,n}v_{z,n}-a_{z,2,n}v_{x,n}||^2+||v_{z,n}-v_{x,n}||^2+||v_{z,n}||^2||v_{x,n}||^2-\langle v_{z,n},v_{x,n}\rangle^2}{\big(1+a_{z,2,n}^2+||v_{z,n}||^2\big)\big(1+a_{x,2,n}^2+||v_{x,n}||^2\big)}. \end{align*} By using the fact that \hbox{$\lim_n \frac{\sigma_1(\rho(\gamma_n))}{\sigma_{2}(\rho(\gamma_n))}(a_{z,2,n}-a_{x,2,n})=g(z)-g(x)\neq 0$} we deduce the bounds: \begin{equation} \label{exponent-eq6} \begin{split}(g(x)-g(z))^2 \frac{\delta^4}{4d^2} \frac{\sigma_2(\rho(\gamma_n))^2}{\sigma_1(\rho(\gamma_n))^2} \leq d_{\mathbb{P}}\big(\xi_{\rho}(\gamma_n z),\xi_{\rho}(\gamma_n x)\big)^2  \leqslant \frac{13d^4}{\delta^4}\frac{\sigma_2(\rho(\gamma_n))^2}{\sigma_1(\rho(\gamma_n))^2} \end{split} \end{equation} for sufficiently large $n \in \mathbb{N}$. Since $z,x\neq \lim_{n}\gamma_{n}^{-1}x_0$, it follows by Lemma \ref{Gromovproduct1} $$\sup_{n\in \mathbb{N}} \Big|(\gamma_n x\cdot \gamma_n z)_{x_0}-|\gamma_n|_{X}\Big|<\infty$$ and hence $\varliminf_{n}d_{a}(\gamma_n x,\gamma_n z) a^{|\gamma_n|_{X}}>0$. Thus, by (\ref{exponent-eq6}) we have $$\alpha \leq \lim_{n \rightarrow \infty} \frac{\log d_{\mathbb{P}}(\xi_{\rho}(\gamma_nx),\xi_{\rho}(\gamma_n z))}{\log d_{a}(\gamma_n x,\gamma_n z)}=\frac{1}{\log a}\lim_{n \rightarrow \infty} \frac{1}{|\gamma_n|_{X}}\log \frac{\sigma_1(\rho(\gamma_n))}{\sigma_2(\rho(\gamma_n))} \leq \beta.$$ This completes the proof of the lemma.\end{proof} 

 Let us recall that for a $1$-Anosov representation $\rho:\Gamma \rightarrow \mathsf{GL}_d(\mathbb{K})$ we set $$\alpha_{\rho}:= \underset{\gamma \in \Gamma_{\infty}}{\inf}\frac{\log \frac{\ell_1}{\ell_2} (\rho(\gamma))}{|\gamma|_{X,\infty}}>0.$$

First, let us prove Corollary \ref{log-bound} which we use in the proof of part of Theorem \ref{exponent}.

\begin{proof}[Proof of Corollary \ref{log-bound}] Fix $w_0,z\in \partial_{\infty}X$ with $z\neq w_0$. Since $\rho$ is $1$-Anosov, by Lemma \ref{mainlemma} (i) and Lemma \ref{triple} we may find $C_0,C_{\rho}>0$ and $0<\varepsilon<1$ depending on $\rho$, with the property: if $x,y\in B_{\varepsilon}(w_0)$ are distinct, there exists $\gamma \in \Gamma$ such that $\gamma^{-1}(x,y,z)\in \mathcal{F}$ and \begin{align}\label{double-ineq-00}||\gamma|_X-(x\cdot y)_{x_0}|\leq C_0, \ d_{\mathbb{P}}\big(\xi_{\rho}^1(x)),\xi_{\rho}^1(y)\big)\leq C_{\rho} \frac{\sigma_2(\rho(\gamma))}{\sigma_1(\rho(\gamma))}.\end{align} By Theorem \ref{semisimple-bound} and (\ref{double-ineq-00}), we may find $C_{\rho}'>0$, depending only on $\rho$, with the property that for every $x,y\in B_{\varepsilon}(w_0)$ and $\gamma \in \Gamma$ with $\gamma^{-1}(x,y,z)\in \mathcal{F}$ we have \begin{align}\label{proper-ineq}\begin{split}  d_{\mathbb{P}}\big(\xi_{\rho}^1(x)),\xi_{\rho}^1(y)\big)& \leq C_{\rho}\frac{\sigma_1(\rho(\gamma))}{\sigma_2(\rho(\gamma))}\leq C_{\rho}' e^{-\alpha_{\rho}|\gamma|_X} |\gamma|_X^{m}\\ &\leq C_{\rho}' e^{-\alpha_{\rho}C_0}e^{-\alpha_{\rho}(x\cdot y)_{x_0}}\big(C_0+(x\cdot y)_{x_0}\big)^{m}\\ & \leq \big(C_{\rho}'L^m e^{-\alpha_{\rho}C_0}r^{\alpha_{\rho}}\big)d_a(x,y)^{\frac{\alpha_{\rho}}{\log a}}\big|\log d_a(x,y)\big|^m,\end{split}\end{align} $L:=2\max\big\{C_0+\frac{\log r}{\log a},\frac{1}{\log a}\big\}$. 

Now, by using the North-South pole dynamics of hyperbolic elements of $\Gamma$, we may choose infinite order elements $h_1,h_2,h_3\in \Gamma$ with the property that for every $x,y\in \partial_{\infty}X$, there exists $i\in \{1,2,3\}$ with $h_i x,h_i y\in B_{\varepsilon}(w_0)$. By Lemma \ref{Gromovproduct1}, there exists $\mathcal{Q}>1$, depending only on $h_{1,2,3}\in \Gamma$ and the matrices $\rho(h_{1,2,3})\in \mathsf{GL}_d(\mathbb{K})$, with the property that for every $y_1,y_2\in \partial_{\infty}X$, \begin{align}\label{Lip-bd}d_a(y_1,y_2)\leq \mathcal{Q}d_a(h_i y_1,h_i y_2), \ d_{\mathbb{P}}\big(\xi_{\rho}^1(h_i y_1),\xi_{\rho}^1(h_i y_2)\big)\leq \mathcal{Q} d_{\mathbb{P}}\big(\xi_{\rho}^1(y_1),\xi_{\rho}^1(y_2)\big).\end{align} If $x,y\in \partial_{\infty}X$ is a pair of distinct points, we may write $x=h_ix',y=h_iy'$, for some $x',y'\in B_{\varepsilon}(w_0)$ and $i\in \{1,2,3\}$, and by (\ref{proper-ineq}) and (\ref{Lip-bd}) we conclude that \begin{align}\label{logbound-1}d_{\mathbb{P}}\big(\xi_{\rho}^1(x),\xi_{\rho}^1(y)\big)\leq Cd_a(x,y)^{\frac{\alpha_{\rho}}{\log a}}\big|\log d_a(x,y)\big|^m\end{align} for $C:=\mathcal{Q}^{\frac{\alpha_{\rho}}{\log a}}(1+\log \mathcal{Q})^mC_{\rho}'L^m e^{-\alpha_{\rho}C_0}r^{\alpha_{\rho}}$.  \end{proof}

\begin{proof}[Proof of Theorem \ref{exponent}.] Let $V=\langle \xi_{\rho}^1(\partial_{\infty}\Gamma)\rangle$. Note that the restriction $\rho|_{V}:\Gamma \rightarrow \mathsf{GL}_d(\mathbb{K})$ is $1$-Anosov with limit map $\xi_{\rho}^1$ and $\ell_1(\rho(\gamma))=\ell_1(\rho|_{V}(\gamma))$ for every $\gamma \in \Gamma$. Thus, we may clearly assume that $\xi_{\rho}^1$ is spanning, i.e. $V=\mathbb{K}^d$.

We split the proof of the theorem in two parts. We first prove that $\xi_{\rho}^1$ is H\"older continuous and $\alpha_{\xi_{\rho}^1}(d_a,d_{\mathbb{P}})=\frac{\alpha_{\rho}}{\log a}$. For this, let $0<\epsilon<d^{-12}\alpha_{\rho}$. By Corollary \ref{log-bound}, fix $C>1$ and $m\in \mathbb{N}$ such that (\ref{logbound-1}) holds for every pair of distinct points $x,y\in \partial_{\infty}X$. In particular, since $d_a$ satisfies (\ref{visualineq}) for every pair of distinct points $x,y\in \partial_{\infty}X$ we have $$\big|\log d_a(x,y)\big|\leq \log r+\log\frac{r}{d_a(x,y)}\leq e^{-1} \epsilon^{-1}r^{2\theta} d_a(x,y)^{-\epsilon}$$ hence $$d_{\mathbb{P}}\big(\xi_{\rho}^1(x),\xi_{\rho}^1(y)\big)\leq C d_a(x,y)^{\frac{\alpha_{\rho}}{\log a}}\big|\log d_a(x,y)\big|^m\leq C \epsilon^{-m} e^{-m} r^{2m\epsilon}d_{a}(x,y)^{\frac{\alpha_{\rho}}{\log a}-m\epsilon}.$$ It follows that $\xi_{\rho}^1$ is $(\frac{\alpha_{\rho}}{\log a}-m\epsilon)$-H\"older. By letting $\epsilon>0$ arbitrarily close to zero, we conclude that $\alpha_{\xi_{\rho}^1}(d_a,d_{\mathbb{P}})\geq \frac{\alpha_{\rho}}{\log a}$.

\par On the other hand, by applying Lemma \ref{doublebound} for the infinite sequence $(\gamma^n)_{n\in \mathbb{N}}$, $\gamma \in \Gamma_{\infty}$, gives the bound $\alpha_{\xi_{\rho}^{1}}(d_a,d_{\mathbb{P}})\leqslant \frac{\alpha_{\rho}}{\log a}.$ This concludes the proof of the first part.
\medskip

Now we prove the second part of the theorem, that if $\rho$ is either irreducible or $2$-Anosov then $\xi_{\rho}^1$ is $\frac{\alpha_{\rho}}{\log a}$-H\"older. Since $\rho$ is either semisimple or $2$-Anosov, by Proposition \ref{12finitesubset}, there exists a finite subset $F\subset \Gamma$ and $D>0$ with the property that for every $\gamma \in \Gamma$ there exists $f \in F$: \begin{equation} \label{exp-eq1} \big||\gamma|_{X}-|\gamma f|_{X,\infty}\big|\leq D,\ \Bigg|\log \frac{\sigma_1(\rho(\gamma))}{\sigma_2(\rho(\gamma))}-\log \frac{\ell_1(\rho(\gamma f))}{\ell_2(\rho(\gamma f))}\Bigg|\leq D.\end{equation} 

\par By Lemma \ref{mainlemma} (i) and Lemma \ref{triple}, we may choose $z,w_0\in \partial_{\infty}X$ distinct and $C_{\rho},C_0, \varepsilon>0$ with the property that if $x,y\in B_{\varepsilon}(w_0)$ and $\gamma^{-1}(x,y,z) \in \mathcal{F}$, then $$\big|(x\cdot y)_{x_0}-|\gamma|_{X}\big|\leq C_0, \ d_{\mathbb{P}}(\xi_{\rho}^1(x),\xi_{\rho}^1(y))\leq C_{\rho}\frac{\sigma_2(\rho(\gamma))}{\sigma_1(\rho(\gamma))}.$$

Let $x,y\in B_{\varepsilon}(w_0)$ and $\gamma \in \Gamma$ such that $\gamma^{-1}(x,y,z)\in \mathcal{F}$ and choose $f\in F$ such that $f,\gamma\in \Gamma$ satisfy (\ref{exp-eq1}). Then we successively obtain the bounds \begin{align*}\frac{d_{\mathbb{P}}\big(\xi_{\rho}^1(x),\xi_{\rho}^1(y)\big)}{d_a(x,y)^{\alpha_{\rho}}}&\leq \frac{C_{\rho}}{d_a(x,y)^{\alpha_{\rho}}}\frac{\sigma_2(\rho(\gamma))}{\sigma_1(\rho(\gamma))} \leq \frac{C_{\rho}e^D}{d_a(x,y)^{\alpha_{\rho}}}\frac{\ell_2(\rho(\gamma f))}{\ell_1(\rho(\gamma f))}\\ &\leq r^{\alpha_{\rho}}C_{\rho}e^D a^{\alpha_{\rho}(x\cdot y)_{x_0}}\frac{\ell_2(\rho(\gamma f))}{\ell_1(\rho(\gamma f))}\\ &\leq r^{\alpha_{\rho}}C_{\rho}e^D a^{C_0\alpha_{\rho}}a^{\alpha_{\rho}|\gamma|_{X}}\frac{\ell_2(\rho(\gamma f))}{\ell_1(\rho(\gamma f))}\\ &\leq r^{\alpha_{\rho}}C_{\rho}e^D a^{(C_0+D)\alpha_{\rho}}a^{\alpha_{\rho}|\gamma f|_{X,\infty}} \frac{\ell_2(\rho(\gamma f))}{\ell_1(\rho(\gamma f))}\\ & \leq  r^{\alpha_{\rho}}C_{\rho}e^D a^{(C_0+D)\alpha_{\rho}}.\end{align*}  

We conclude that $\xi_{\rho}^1$ is $\frac{\alpha_{\rho}}{\log a}$-H\"older restricted on the open ball $B_{\varepsilon}(w_0)$. As previously, we can choose infinite order elements $h_1,h_2,h_3\in \Gamma$ such that for every $x,y\in \partial_{\infty}X \smallsetminus B_{\varepsilon}(w_0)$, there exists $i\in \{1,2,3\}$ with $h_i x,h_i y\in B_{\varepsilon}(w_0)$. Since $\rho(h_i)$ (resp. $h_i$) is bi-Lipschitz with respect to the metric $d_{\mathbb{P}}$ (resp. $d_a$), we conclude that \hbox{$\xi_{\rho}^1$ is $\frac{\alpha_{\rho}}{\log a}$-H\"older.} \end{proof} 

Recall that for a linear real semisimple Lie group $G$, $\lambda:G\rightarrow \overline{\mathfrak{a}}^{+}$ denotes the {\em Lyapunov projection} defined as follows $\lambda(g)=\lim_n \frac{\mu(g^n)}{n}$, $g\in G.$

\begin{proof}[Proof of Theorem \ref{Exp-general}.] By \cite[Prop. 3.5]{GGKW}, there exists an irreducible representation $\tau:G \rightarrow \mathsf{GL}_d(\mathbb{R})$ such that $\tau \circ \rho:\Gamma \rightarrow \mathsf{GL}_d(\mathbb{R})$ is irreducible and $1$-Anosov. By the definition of the metrics $d_{\theta^{+}}$ and $d_{\theta^{-}}$ (see the discussion in subsection \ref{generalcase}) we have \begin{align*} \alpha_{\xi_{\rho}^{+}}(d_a,d_{\theta^{+}})=\alpha_{\xi_{\tau \circ \rho}^1}(d_a,d_{\mathbb{P}}),\ \alpha_{\xi_{\rho}^{-}}(d_a,d_{\theta^{-}})=\alpha_{\xi_{\tau \circ \rho}^{\ast}}(d_a,d_{\mathbb{P}})\end{align*} where $\xi_{\tau \circ \rho}^{\ast}:\partial_{\infty}X\rightarrow \mathbb{P}(\mathbb{R}^d)$ is the Anosov limit map of the dual representation $(\tau \circ \rho)^{\ast}:\Gamma \rightarrow \mathsf{GL}_d(\mathbb{R})$. Moreover, by \cite[Lem. 3.7]{GGKW}, we have the following relation for the Lyapunov projection $$\min_{\varphi \in \theta}\varphi(\lambda(\rho(\gamma)))=\log \frac{\ell_1}{\ell_2}\big(\tau(\rho(\gamma))\big) \ \ \forall \gamma \in \Gamma,$$ and the conclusion follows by Theorem \ref{exponent}.\end{proof}

Given a metric space $(Y,d_Y)$ denote by $\textup{dim}(Y,d_Y)$ its Hausdorff dimension and by $\textup{dim}_{\textup{top}}(Y)$ its topological dimension. Let $F_r$, $r\geq 2$, be the free group on $r$ generators. For a $1$-Anosov representation $\psi$ of $F_r$ the H$\ddot{\textup{o}}$lder exponent $\alpha_{\xi_{\psi}^{1}}(d_{a},d_{\mathbb{P}})$ can be arbitrarily large. To see this, let $\{a_1,\ldots,a_r\}$ be a  free generating subset of $F_r$ and fix $\psi:F_r\rightarrow \mathsf{SL}_2(\mathbb{R})$ a $1$-Anosov representation. The sequence of representations $\{\psi_n:F_r\rightarrow \mathsf{SL}_2(\mathbb{R})\}_{n\in \mathbb{N}}$, where $\psi_n(a_i)=\psi(a_i)^n$ for every $1\leq i \leq r$, clearly satisfies $\lim_n \alpha_{\xi_{\psi_n}^1}(d_a,d_{\mathbb{P}})=+\infty$.

\par Now suppose that $\Gamma$ is a hyperbolic group which is not virtually free and $\rho:\Gamma \rightarrow \mathsf{GL}_d(\mathbb{K})$ is a $1$-Anosov representation. As $\Gamma$ (and $\Gamma/\textup{ker}\rho$) are not virtually free, it follows by \cite{Stallings-free} and \cite{Bestvina-Mess} that $\textup{dim}_{\textup{top}}(\partial_{\infty}\Gamma)>0$. By the definition of Hausdorff dimension, then one verifies the upper bound $$\alpha_{\xi_{\rho}^{1}}(d_{a},d_{\mathbb{P}}) \textup{dim} (\xi_{\rho}^{1}(\partial_{\infty}X),d_{\mathbb{P}}) \leqslant  \textup{dim}(\partial_{\infty}X,d_{a}).$$ Recall that $\Gamma$ acts geometrically on the hyperbolic space $(X,d_X)$. The visual boundary $\partial_{\infty}X$, equipped with the metric $d_a$, satisfies $\textup{dim}(\partial_{\infty}X,d_a)=\frac{h_{\Gamma,X}}{\log a}$, where $$h_{\Gamma,X}=\lim_{R\rightarrow \infty}\frac{1}{R}\log \big|\big\{\gamma \in \Gamma:|\gamma|_X\leq R\big\}\big|$$ see \cite[Thm. 7.7]{Coornaert}. Since $\xi_{\rho}^{1}$ is a homeomorphism, $\textup{dim}(\xi_{\rho}^{1}(\partial_{\infty}X),d_{\mathbb{P}})\geq \textup{dim}_{\textup{top}}(\partial_{\infty}X)$ (e.g. see \cite[Thm. 6.3.11]{Edgar}) and hence $$\sup \big\{ \alpha_{\xi_{\rho}^1}(d_a,d_{\mathbb{P}}):\rho \ \textup{is} \ \textup{1-Anosov}\big\}\leq \frac{1}{\log a}\frac{h_{\Gamma,X}}{\textup{dim}_{\textup{top}}(\partial_{\infty}X)}.$$ In particular, as a corollary of Theorem \ref{exponent} and the previous remarks we obtain: 

\begin{corollary} Suppose that $\rho:\Gamma \rightarrow \mathsf{GL}_d(\mathbb{K})$ is a $1$-Anosov representation whose limit map in $\mathbb{P}(\mathbb{K}^d)$ is spanning. For every $\epsilon>0$ there exists an infinite order element $\gamma \in \Gamma$ such that $$\log\frac{\ell_1(\rho(\gamma))}{\ell_2(\rho(\gamma))}\leq \frac{ (1+\epsilon)h_{\Gamma,X}}{\textup{dim}_{\textup{top}}(\partial_{\infty}X)}|\gamma|_{X,\infty}.$$\end{corollary}

\section{H\:older exponent of the inverse limit map} \label{inverse}

In this section, we prove Theorem \ref{exponent-inverse2} that the inverse of an injective limit map is H\"older continuous as well as the formulas claimed in Theorem \ref{exponent2} and Theorem \ref{conjugation}.

Before we give the proofs, we shall make some conventions that will be useful. Given a model space $(X,d_x)$ for a hyperbolic group $\Gamma$, fix $w_0\in \partial_{\infty}X$ and $z\neq w_0$. Let also $C,\varepsilon>0$, depending only on $\Gamma$ and $z\in \partial_{\infty}X$, satisfying the conclusion of Lemma \ref{triple}: if $x,y\in B_{\varepsilon}(w_0)$, $x\neq y$, and $\gamma \in \Gamma$ such that $\gamma^{-1}(x,y,z)\in \mathcal{F}$, \begin{align} \label{ineq-Grproduct-wordlength} |\gamma|_X-C_0\leq (x\cdot y)_{x_0}\leq |\gamma|_{X}+C_0.\end{align}

\begin{proof}[Proof of Theorem \ref{exponent-inverse2}]  Let $\rho:\Gamma \rightarrow \mathsf{GL}_d(\mathbb{K})$ be a representation which admits a continuous, $\rho$-equivariant, injective map $\xi_{\rho}:\partial_{\infty}X \rightarrow \mathbb{P}(\mathbb{K}^d)$. We prove that there exist $\beta,\kappa>0$ such that \begin{align}\label{betaconv} d_{\mathbb{P}}(\xi_{\rho}(x),\xi_{\rho}(y))\geq \kappa d_a(x,y)^{\beta} \ \ \forall \  x,y\in \partial_{\infty}X.\end{align}

\par By the minimality of the action of $\Gamma$ on $\partial_{\infty}X$, it suffices to prove that (\ref{betaconv}) holds into any open neighbourhood $V\subset \partial_{\infty}X$ of $w_0\in \partial_{\infty}X$. Since $\Gamma$ is finitely generated, by the sub-multiplicativity of the first singular value, we may choose a constants $C,c>1$ with $$ \frac{\sigma_1(\rho(\gamma))^2}{\sigma_d(\rho(\gamma))\sigma_{d-1}(\rho(\gamma))}\leq C e^{c |\gamma|_{X}} \ \ \forall \ \gamma \in \Gamma.$$ If $x,y\in B_{\varepsilon}(w_0)$ and $(x,y,z)=\gamma(b_1,b_2,b_3)\in \gamma \mathcal{F}$, by Lemma \ref{mainlemma2'} and (\ref{ineq-Grproduct-wordlength}), we obtain the bounds: \begin{align*} d_{\mathbb{P}}\big(\xi_{\rho}(x),\xi_{\rho}(y)\big)&\geq \frac{2}{\pi}\frac{\sigma_d(\rho(\gamma))\sigma_{d-1}(\rho(\gamma))}{\sigma_1(\rho(\gamma))^2}d_{\mathbb{P}}\big(\xi_{\rho}(b_1),\xi_{\rho}(b_2)\big)\\  &\geq \frac{2}{\pi C}e^{-c|\gamma|_{X}}\min_{(x_1,x_2,x_3)\in \mathcal{F}}d_{\mathbb{P}}\big(\xi_{\rho}(x_1),\xi_{\rho}(x_2)\big)\\ &\geq \frac{2}{\pi C}e^{-c(x\cdot y)_{x_0}}e^{-c C_0}\min_{(x_1,x_2,x_3)\in \mathcal{F}}d_{\mathbb{P}}\big(\xi_{\rho}(x_1),\xi_{\rho}(x_2)\big)\\ & \geq \Bigg(\frac{2}{\pi C} r^{-\frac{c}{\log a}} e^{-c C_0}\min_{(x_1,x_2,x_3)\in \mathcal{F}}d_{\mathbb{P}}\big(\xi_{\rho}(x_1),\xi_{\rho}(x_2)\big)\Bigg) d_a(x,y)^{\frac{c}{\log a}}.\end{align*}  This shows that inverse of $\xi_{\rho}$ is H\"older continuous and concludes the proof of the theorem.
\end{proof}

\begin{proof}[Proof of Theorem \ref{exponent2}.]  Let $\rho$ be a $(1,1,2)$-hyperconvex whose Anosov limit map $\xi_{\rho}^1$ is spanning. Let us define \begin{align*} \beta(d_a,d_{\mathbb{P}})&:=\inf \big \{\beta>0 \ \big| \ \exists \ \kappa>0: d_{\mathbb{P}}(\xi_{\rho}^{1}(x),\xi_{\rho}^{1}(y)) \geqslant  \kappa d_{a}(x,y)^{\beta} \ \forall \ x,y \in \partial_{\infty} X \big\} \end{align*} and recall that we set $\beta_{\rho}=\sup_{\gamma \in \Gamma_{\infty}}\frac{\log\frac{\ell_1}{\ell_2}(\rho(\gamma))}{|\gamma|_{X,\infty}}$. 

\par Note that $\beta(d_a,d_{\mathbb{P}})>0$ and the H\"older exponent of the inverse of $\xi^{1}_{\rho}$ is equal to $\beta(d_a,d_{\mathbb{P}})^{-1}$. Since $\xi_{\rho}^1$ is spanning, by using Lemma \ref{doublebound} it is immediate that $\beta(d_a,d_{\mathbb{P}}) \geq \frac{\beta_{\rho}}{\log a}$. We will prove that there exists $\kappa>0$: $$d_{\mathbb{P}}\big(\xi_{\rho}^{1}(x),\xi_{\rho}^{1}(y) \big) \geqslant  \kappa d_{a}(x,y)^{\frac{\beta_{\rho}}{\log a}} \ \ \forall x,y \in \partial_{\infty} X.$$ This is enough to conclude $\beta(d_a,d_{\mathbb{P}})=\frac{\beta_{\rho}}{\log a}$ and that $\beta(d_a,d_{\mathbb{P}})$ is attained.

Let $x,y\in B_{\varepsilon}(w_0)$ with $x\neq y$ and choose $\gamma \in \Gamma$ such that $\gamma^{-1}(x,y,z)\in \mathcal{F}$ and \hbox{$|\gamma|_X \leq (x\cdot y)_{x_0}+C_0$.} By using the lower bound of Theorem \ref{semisimple-bound} and Lemma \ref{mainlemma} (ii) there exist $L_{\rho}>1$ and $c_{\rho}>0$, depending only on $\rho$ and $\Gamma$, such that: \begin{align*}d_{\mathbb{P}}\big(\xi_{\rho}^1(x),\xi_{\rho}^1(y) \big) &\geq c_{\rho} \frac{\sigma_2(\rho(\gamma))}{\sigma_1(\rho(\gamma))}\\ &\geq c_{\rho} L_{\rho}^{-1} e^{-\beta_{\rho}|\gamma|_{X}}\\ &\geq c_{\rho} L_{\rho}^{-1} e^{-\beta_{\rho}C_0}e^{-\beta_{\rho}(x\cdot y)_{x_0}}\\ &\geq \big(c_{\rho} L_{\rho}^{-1}e^{-\beta_{\rho}C_0}r^{-\frac{\beta_{\rho}}{\log a}}\big) d_a(x,y)^{\frac{\beta_{\rho}}{\log a}}. \end{align*} Therefore, $\beta(d_a,d_{\mathbb{P}})=\frac{\beta_{\rho}}{\log a}$ and the inverse $\eta_{\rho}^1:(\xi_{\rho}^1(\partial_{\infty}X),d_{\mathbb{P}})\rightarrow (\partial_{\infty}X,d_a)$ of the limit map $\xi_{\rho}^1$ is $\frac{\log a}{\beta_{\rho}}$-H\"older. \end{proof}

\begin{proof}[Proof of Theorem \ref{conjugation}.] Let us observe that the H\"older exponent of $\xi_{\rho_2}^{1}\circ (\xi_{\rho_1}^{1})^{-1}$ \hbox{is equal to} $$a_{\rho_1,\rho_2}:= \sup \big \{ \beta>0 \ |\ \exists \ \kappa>0: d_{\mathbb{P}}(\xi_{\rho_2}^{1}(x),\xi_{\rho_2}^{1}(y)) \leqslant \kappa d_{\mathbb{P}}(\xi_{\rho_1}^{1}(x),\xi_{\rho_1}^{1}(y))^{\beta} \ \forall x,y \in \partial_{\infty} X \big\}.$$ Note that $a_{\rho_1,\rho_2}>0$ since the limit maps $\xi_{\rho_1}^1$ and $\xi_{\rho_2}^1$ are bi-H\"older continuous. By Corollary \ref{equality-exp-1}, since $\rho_1$ is $\{1,2\}$-Anosov, we have the equality $$\alpha_{\rho_1,\rho_2}:=\inf_{\gamma \in \Gamma_{\infty}} \frac{\log \frac{\ell_1}{\ell_2}(\rho_2(\gamma))}{\log \frac{\ell_1}{\ell_2}(\rho_1(\gamma))}=\sup_{n \geq 1} \inf_{|\gamma|_{X}\geq n} \frac{\log \frac{\sigma_1}{\sigma_2}(\rho_2(\gamma))}{\log \frac{\sigma_1}{\sigma_2}(\rho_1(\gamma))}.$$ 

We prove the lower bound $a_{\rho_1,\rho_2}\geq \alpha_{\rho_1,\rho_2}$. For this, let $\epsilon>0$. By the definition of $a_{\rho_1,\rho_2}>0$, we may choose infinite sequences $(z_n)_{n \in \mathbb{N}}$ and $(w_n)_{n \in \mathbb{N}}$ in $\partial_{\infty}X$ such that $z_n \neq w_n$ for every $n$, $\lim_n z_n=\lim_n w_n=x'$ and \begin{equation} \label{conjugation-eq0} d_{\mathbb{P}}\big(\xi_{\rho_2}^{1}(z_n),\xi_{\rho_2}^{1}(w_n) \big) \geqslant d_{\mathbb{P}}\big(\xi_{\rho_1}^{1}(z_n),\xi_{\rho_1}^{1}(w_n) \big)^{a_{\rho_1,\rho_2}+\epsilon}\end{equation} for sufficiently large $n$. 

Let us fix $z'\in \partial_{\infty}X \smallsetminus \{x'\}$.  Since $\rho_1$ is $(1,1,2)$-hyperconvex and $\rho_2$ is $1$-Anosov, by Lemma \ref{mainlemma} (i) and (ii) there exist $\varepsilon>0$ and $c_{\rho_1},C_{\rho_2}>0$ depending on $\rho_1$ and $\rho_2$ respectively, with the property that if $x,y\in B_{\varepsilon}(x')$ and $\delta \in \Gamma$ is an element with $\delta^{-1}(x,y,z')\in \mathcal{F}$, then: \begin{align} \label{conjugation-eq1}d_{\mathbb{P}}\big(\xi_{\rho_1}^{1}(x),\xi_{\rho_1}^{1}(y)\big)  &\geq c_{\rho_1} \frac{\sigma_2(\rho_1(\delta))}{\sigma_1(\rho_1(\delta))} \\
\label{conjugation-eq2} d_{\mathbb{P}}\big(\xi_{\rho_2}^{1}(x),\xi_{\rho_2}^{1}(y)\big) &\leq C_{\rho_2} \frac{\sigma_2(\rho_2(\delta))}{\sigma_1(\rho_2(\delta))}. \end{align} Now for sufficiently large $n$, let $\delta_n\in \Gamma$ with $\delta_{n}^{-1}(z_n,w_n,z')\in \mathcal{F}$. It follows by (\ref{conjugation-eq0}), (\ref{conjugation-eq1}) and (\ref{conjugation-eq2}) $$a_{\rho_1,\rho_2}+\epsilon \geq \varliminf_{n\rightarrow \infty} \frac{\log \frac{\sigma_1}{\sigma_2}(\rho_2(\delta_n))}{\log \frac{\sigma_1}{\sigma_2}(\rho_1(\delta_n))} \geq \alpha_{\rho_1,\rho_2}.$$ Since $\epsilon>0$ was arbitrary we conclude that $a_{\rho_1,\rho_2}\geq \alpha_{\rho_1,\rho_2}$.

\par Now we prove the upper bound $a_{\rho_1,\rho_2}\leq \alpha_{\rho_1,\rho_2}$. Fix $\beta>0$ so that there exists $c>0$ with \begin{equation} \label{conjugation-eq0''} d_{\mathbb{P}}(\xi_{\rho_2}^{1}(x),\xi_{\rho_2}^{1}(y)) \leqslant c  d_{\mathbb{P}}(\xi_{\rho_1}^{1}(x),\xi_{\rho_1}^{1}(y))^{\beta} \ \ \forall x,y\in \partial_{\infty}X.\end{equation} Let $(\gamma_n')_{n \in \mathbb{N}}$ be an infinite sequence in $\Gamma$. By using the fact that the Anosov limit maps of $\rho_1$ and $\rho_2$ are spanning and working as in Lemma \ref{doublebound}, we may choose $x_1,x_2 \in \partial_{\infty}\Gamma$ such that for $i=1,2,$ $$\lim_{n\rightarrow \infty}\frac{\log d_{\mathbb{P}}\big(\xi_{\rho_i}^1(\gamma_n' x_1),\xi_{\rho_i}^{1}(\gamma_n' x_2)\big)}{ \log \frac{\sigma_2}{\sigma_1}(\rho_i(\gamma_n'))}=1.$$ It follows by (\ref{conjugation-eq0''}) that \begin{equation} \label{conjugation-eq3} \beta \leq \lim_{n \rightarrow \infty} \frac{\log d_{\mathbb{P}}\big(\xi_{\rho_2}^1(\gamma_n'x_1),\xi_{\rho_2}^{1}(\gamma_n' x_2)\big)}{\log d_{\mathbb{P}}\big(\xi_{\rho_1}^1(\gamma_n'x_1),\xi_{\rho_1}^{1}(\gamma_n' x_2)\big)}=\lim_{n \rightarrow \infty}\frac{\log \frac{\sigma_1}{\sigma_2}(\rho_2(\gamma_n'))}{\log \frac{\sigma_1}{\sigma_2}(\rho_1(\gamma_n'))}.\end{equation} In particular, $\beta \leq \alpha_{\rho_1,\rho_2}$ and hence $a_{\rho_1,\rho_2}\leq \alpha_{\rho_1,\rho_2}$. Finally, we conclude that $a_{\rho_1,\rho_2}=\alpha_{\rho_1,\rho_2}$. 

It remains to show that if $\rho_2$ is either irreducible or $2$-Anosov, then $\alpha:=\alpha_{\rho_1,\rho_2}$ is attained. By Proposition \ref{12finitesubset}, there exists a finite set $F\subset \Gamma$ and $L>0$ with the property that for every $\gamma \in \Gamma$ there exists $f \in F$ with \begin{equation} \label{exp-eq2} \max_{i=1,2} \Bigg|\log \frac{\sigma_1(\rho_i(\gamma))}{\sigma_2(\rho_i(\gamma))}-\log \frac{\ell_1(\rho_i(\gamma f))}{\ell_2(\rho_i(\gamma f))}\Bigg|\leq L. \end{equation}

Let $x,y\in B_{\varepsilon}(x')$ and $\gamma \in \Gamma$ such that $\gamma^{-1}(x,y,z')=(b_1,b_2,b_3)\in \mathcal{F}$. We may choose $f\in F$ so that $\gamma,f\in \Gamma$ satisfy (\ref{exp-eq2}) and by using (\ref{conjugation-eq1}) and (\ref{conjugation-eq2}) we obtain the estimates: \begin{align*} d_{\mathbb{P}}\big(\xi_{\rho_2}^1(x),\xi_{\rho_2}^1(y) \big)&=d_{\mathbb{P}}\big(\rho_2(\gamma)\xi_{\rho_2}^1(b_1),\rho_2(\gamma)\xi_{\rho_2}^1(b_2) \big)\\ &\leq C_{\rho_2}\frac{\sigma_2(\rho_2(\gamma))}{\sigma_1(\rho_2(\gamma))}\leq C_{\rho_2}e^{L}\frac{\ell_2(\rho_2(\gamma f))}{\ell_1(\rho_2(\gamma f))}\\ &\leq C_{\rho_2}e^L  \frac{\ell_2(\rho_1(\gamma f))^{\alpha}}{\ell_1(\rho_1(\gamma f))^{\alpha}}\leq C_{\rho_2}e^{L(1+\alpha)} \frac{\sigma_2(\rho_1(\gamma))^{\alpha}}{\sigma_1(\rho_1(\gamma))^{\alpha}}\\ &\leq C_{\rho_2}c_{\rho_1}^{-\alpha} e^{L(1+\alpha)}d_{\mathbb{P}}\big(\xi_{\rho_1}^1(x),\xi_{\rho_1}^1(y) \big)^{\alpha}.\end{align*} This shows that the map $\xi_{\rho_1}^{1}\circ (\xi_{\rho_2}^1)^{-1}$ is $\alpha_{\rho_1,\rho_2}$-H\"older restricted on the open subset $\xi_{\rho_2}^1(B_{\varepsilon}(x'))$ of $\xi_{\rho_2}^1(\partial_{\infty}X)$. In particular, since $\rho_2(\Gamma)$ acts minimally on $\xi_{\rho_2}^1(\partial_{\infty}X)$ we conclude that $\xi_{\rho_1}^{1}\circ (\xi_{\rho_2}^1)^{-1}$ is is $\alpha_{\rho_1,\rho_2}$-H\"older on $\xi_{\rho_2}^1(\partial_{\infty}X)$.\end{proof}

Let $\textup{Hit}_d(\Sigma_g)$ be the Hitchin component of $\pi_1(\Sigma_g)$ into $\mathsf{PSL}_d(\mathbb{R})$. For every $r=1,\ldots,d-1$, Carvajales--Dai--Pozzetti--Wienhard in \cite{CDPW} defined an asymmetric distance $d_{\textup{Th}}^r: \textup{Hit}_d(\Sigma_g) \times \textup{Hit}_d(\Sigma_g)\rightarrow \mathbb{R}$, generalizing Thurston's asymmetric metric on Teichm\"uller space of $\Sigma_g$. The calculation of Theorem \ref{conjugation}, combined with the fact that exterior powers of Hitchin representations are $(1,1,2)$-hyperconvex \cite[Prop. 9.6]{PSW} and Theorem 1.3 of \cite{CDPW}, imply the following corollary.

\begin{corollary} Let $\rho_1,\rho_2\in \textup{Hit}_d(\Sigma_g)$ and an integer $2\leq r\leq d-1$.\\
\noindent \textup{(i)} The map $\xi_{\rho_2}^1\circ (\xi_{\rho_1}^{1})^{-1}:\xi_{\rho_1}^1(\partial_{\infty}\pi_1(\Sigma_g))\rightarrow \xi_{\rho_2}^1(\partial_{\infty}\pi_1(\Sigma_g))$ has H\"older exponent equal to $1$ if and only if $\rho_1$ and $\rho_2$ are conjugate.\\
\noindent \textup{(ii)} Suppose that $\rho_1$ and $\rho_2$ are Zariski dense and set $\psi_j:=\wedge^r\rho_j:\pi_1(\Sigma_g)\rightarrow \mathsf{PSL}(\wedge^r \mathbb{R}^d)$ for $j=1,2$. The map $\xi_{\psi_2}^1\circ (\xi_{\psi_1}^{1})^{-1}:\xi_{\psi_1}^1(\partial_{\infty}\pi_1(\Sigma_g))\rightarrow \xi_{\psi_2}^1(\partial_{\infty}\pi_1(\Sigma_g))$ has H\"older exponent equal to $1$ if and only if $\rho_1$ and $\rho_2$ are conjugate.\end{corollary}

\subsection{The rank 1 case.} We close this section with the following proposition for the H\"older exponent of limit maps of Anosov representations into rank one Lie groups. Let $G$ be a real semisimple Lie group with $\textup{rk}_{\mathbb{R}}(G)=1$. Consider the symmetric space of $G$, $\mathsf{X}_G$ equipped with the distance $d_{\mathsf{X}_G}$ induced by the Killing metric. Fix also a visual metric $d_{b}$, $b>1$, on $\partial_{\infty}\mathsf{X}_G$ such that are $L_1,L_2>1$ with $$L_1b^{-(x\cdot y)_{x_0}}\leq d_{b}(x,y)\leq L_2 b^{-(x\cdot y)_{x_0}},\  x,y\in \partial_{\infty}\mathsf{X}_{G}.$$ 

\begin{proposition} Let $G$ be a real semisimple Lie group with $\textup{rk}_{\mathbb{R}}(G)=1$ and $\rho:\Gamma \rightarrow G$ a Zariski dense Anosov representation with limit map $\xi_{\rho}:(\partial_{\infty}X,d_a)\rightarrow (\partial_{\infty}\mathsf{X}_G,d_b)$. The exponent of $\xi_{\rho}$ and its inverse $\xi_{\rho}^{-1}$ are attained and $$\alpha_{\xi_{\rho}}(d_a,d_{b})=\frac{\log b}{\log a}c_{\rho}^{-},\ \alpha_{\xi_{\rho}^{-1}}(d_{b},d_a)=\frac{\log a}{\log b}\frac{1}{c_{\rho}^{+}}$$ where $c_{\rho}^{-}:=\underset{\gamma\in \Gamma_{\infty}}{\inf}\frac{|\rho(\gamma)|_{\mathsf{X}_G,\infty}}{|\gamma|_{X,\infty}},\ c_{\rho}^{+}:=\underset{\gamma\in \Gamma_{\infty}}{\sup}\frac{|\rho(\gamma)|_{\mathsf{X}_G,\infty}}{|\gamma|_{X,\infty}}.$\end{proposition}

\begin{proof} The Anosov representation $\rho:\Gamma \rightarrow G$ is {\em convex cocompact}, i.e. $\rho(\Gamma)$ acts cocompactly on the convex hull $\mathcal{C}_{\rho}\subset \mathsf{X}_G$ of its limit set in $\mathsf{X}_G$ (e.g. see \cite[Thm. 1.8]{GW} and \cite[Thm. 11.1]{Canary}). Fix basepoints $x_0\in X$ and $y_0'\in \mathcal{C}_{\rho}$. The limit map of $\xi_{\rho}$ is identified with the induced boundary map of the quasi-isometry orbit map $(\Gamma x_0,d_X)\rightarrow (\mathcal{C}_{\rho}, d_{\mathsf{X}_G})$, $\gamma x_0\mapsto \rho(\gamma)y_0'$. In other words, if $(\gamma_n)_{n\in \mathbb{N}}\subset \Gamma$ is an infinite sequence converging to a point in $\partial_{\infty}X$, then $\xi_{\rho}(\lim_{n} \gamma_n)=\lim_{n} \rho(\gamma_n)y_0'$. By \cite[Prop. 3.5.4]{thesis}, there is $C_1>1$ such that for every $\gamma, \delta \in \Gamma$ we have \begin{align*} c_{\rho}^{-}(\gamma x_0\cdot \delta x_0)_{x_0}-C_1&\leq \big(\rho(\gamma)y_0'\cdot \rho(\delta)y_0'\big)_{y_0'}\leq c_{\rho}^{+} (\gamma x_0\cdot \delta x_0)_{x_0}+C_1.\end{align*} By the definition of $\xi_{\rho}$ and the metrics $d_b,d_a$, there is $C_2>1$ such that for every $x,y\in \partial_{\infty}X$: \begin{align}\label{rank1ineq}C_2^{-1}d_a(x,y)^{\frac{\log b}{\log a}c_{\rho}^{+}}\leq d_b\big(\xi_{\rho}(x),\xi_{\rho}(y)\big)\leq C_2 d_a(x,y)^{\frac{\log b}{\log a}c_{\rho}^{-}}.\end{align} By Proposition \ref{Gromovproduct1} we have $\alpha_{\xi_{\rho}}(d_a,d_b)\leq \frac{\log b}{\log a}c_{\rho}^{-}$ and $\alpha_{\xi_{\rho}^{-1}}(d_b,d_a)\leq \frac{\log a}{\log b}\frac{1}{c_{\rho}^{+}}$. Then (\ref{rank1ineq}) shows that exponent of the limit maps of $\xi_{\rho}$ and its inverse $\xi_{\rho}^{-1}$ are attained and are equal to $\frac{\log b}{\log a}c_{\rho}^{-}$ and $\frac{\log a}{\log b}\frac{1}{c_{\rho}^{+}}$ respectively.\end{proof}

\section{Anosov limit maps which do not attain their H\"older exponent}\label{1-Holder-proof}

In this section, we prove Theorem \ref{1-Holder} by constructing families of $1$-Anosov representations of surface groups in $\mathsf{SL}_4(\mathbb{R})$ whose Anosov limit maps in $\mathbb{P}(\mathbb{R}^4)$ do not attain their H\"older exponent. 

 \begin{proof}[Proof of Theorem \ref{1-Holder}] Fix a basepoint $x_0\in \mathbb{H}_{\mathbb{R}}^2$, $d_{\mathbb{H}^2}$ the standard Riemannian distance on $\mathbb{H}_{\mathbb{R}}^2$ and the visual metric $$d_\textup{v}(x,y)=e^{-(x\cdot y)_{x_0}},\ x,y\in \partial_{\infty}\mathbb{H}_{\mathbb{R}}^2.$$ We also fix a discrete faithful representation $\rho_1:\pi_1(\Sigma_g)\rightarrow \mathsf{SL}_2(\mathbb{R})$ realizing the hyperbolic plane $(\mathbb{H}_{\mathbb{R}}^2,d_{\mathbb{H}^2})$ as a model space for $\pi_1(\Sigma_g)$, and providing an equivariant identification $\partial_{\infty}\pi_1(\Sigma_g)\cong \partial_{\infty}\mathbb{H}_{\mathbb{R}}^2$.
\medskip

\noindent {\em Construction of the family of representations $\big\{\rho_{s,t}:\pi_1(\Sigma_g)\rightarrow \mathsf{SL}_4(\mathbb{R})\big\}_{(s,t)\in \mathcal{O}}, \mathcal{O}:=\mathbb{R}\times (-\varepsilon,\varepsilon).$}

Fix the following presentation of $\pi_1(\Sigma_g)$: $$\pi_1(\Sigma_g)=\Big \langle a_1,b_1,\ldots, a_g,b_g \ \Big| \ [a_1,b_1]\cdots [a_g,b_g] \Big \rangle.$$ For $t\in \mathbb{R}$ let $J_{t}:=\begin{pmatrix}[0.9]
1&t \\ 
0 & 1
\end{pmatrix}$, fix a matrix $A\in \mathsf{GL}_2(\mathbb{R})$ and define: \begin{align} \label{2-matrices} A_{s,t}:=s\big(A\rho_1(b_1)^{-1}-J_tA\rho_1(a_1^{-1}b_1^{-1})\big)\rho_1\big(a_2b_2a_2^{-1}\big)\big(\rho_1(a_2^{-1})-\textup{I}_2\big)^{-1}. \end{align} 
Clearly $A_{s,t}\in \textup{Mat}_2(\mathbb{R})$ is well defined since the eigenvalues of $\rho_1(a_2^{-1})$ are different from $1$. Moreover, let $\psi_{2,t}:\pi_1(\Sigma_g) \rightarrow \mathsf{SL}_{2}(\mathbb{R})$ be the representation defined on the set of generators $\{a_1,b_1,\ldots,a_g,b_g\}$: $$\psi_{2,t}(a_1)=J_t, \ \psi_{2,t}(b_1)=\psi_{2,t}(a_i)=\psi_{2,t}(b_i)=\textup{I}_2,\ i=2,\ldots,g.$$ Now consider the unique representation $\rho_{s,t}:\pi_1(\Sigma_g)\rightarrow \mathsf{SL}_4(\mathbb{R})$ satisfying: \begin{align*}
\rho_{s,t}(a_i)&=\begin{pmatrix}
\psi_{2,t}(a_i) & \\ 
 & \rho_1(a_i)
\end{pmatrix} \ i=1,2,\ldots, g \\ \rho_{s,t}(b_j)&=\begin{pmatrix}
\psi_{2,t}(b_j) & \\ 
 & \rho_1(b_j)
\end{pmatrix} \ j=3,\ldots,g \ (\textup{if} \ g>2)\\
\rho_{s,t}(b_1)&=\begin{pmatrix}[1]
\psi_{2,t}(b_1) & sA \\ 
 & \rho_1(b_1)
\end{pmatrix}, \ \rho_{s,t}(b_2)=\begin{pmatrix}
\psi_{2,t}(b_2) & A_{s,t}\\ 
 & \rho_1(b_2)
\end{pmatrix}.\end{align*}

A straightforward calculation using (\ref{2-matrices}) shows that \begin{align*}\rho_{s,t}([a_1,b_1][a_2,b_2])&=\begin{pmatrix}[0.9]
\textup{I}_2 & B_{s,t} \\ 
 & \rho_1([a_1,b_1][a_2,b_2])
\end{pmatrix},\\ B_{s,t}:=A_{s,t}(\rho_1(a_2^{-1})-\textup{I}_2)\rho_1(b_2^{-1})+&s\big(J_t A\rho_1(a_1^{-1}b_1^{-1})-A\rho_1(b_1^{-1})\big)\rho_1(a_2b_2a_2^{-1}b_2^{-1})=\begin{pmatrix} 0  & 0 \\ 0 & 0 \end{pmatrix}\end{align*} Therefore, $\rho_{s,t}([a_1,b_1]\cdots [a_g,b_g])=\textup{I}_4$ and $\rho_{s,t}$ is a well defined representation of $\pi_1(\Sigma_g)$.

\medskip
\noindent {\em Claim 1. The representation $\rho_{s,t}:\pi_1(\Sigma_g)\rightarrow \mathsf{SL}_4(\mathbb{R})$ is $1$-Anosov. There exists $\varepsilon>0$ such that the Anosov limit map $\xi_{s,t}^1:\partial_{\infty}\mathbb{H}_{\mathbb{R}}^2\rightarrow \mathbb{P}(\mathbb{R}^4)$ is spanning for every $|t|<\varepsilon$ and $s\neq 0$.}

Since the image of $\rho_{2,t}$ is unipotent, the moduli of the eigenvalues of \hbox{$\rho_{s,t}(\gamma)$ are} $$\ell_1(\rho_1(\gamma)),1,1, \ell_1(\rho_1(\gamma))^{-1}$$ and hence $\rho$ is $1$-Anosov by Theorem \ref{Anosov} (iii).

For every $s,t\neq 0$ the vector space $V_{s,t}\subset \mathbb{R}^4$ spanned by $\xi_{s,t}^1(\partial_{\infty}\mathbb{H}_{\mathbb{R}}^2)$ contains $\mathbb{R}e_3\oplus \mathbb{R}e_4$, since the restriction of $\xi_{s,t}^1$ on the limit set of the free group $\langle \rho_1(a_1),\rho_1(a_2)\rangle$ in $\partial_{\infty}\mathbb{H}_{\mathbb{R}}^2$ coincides with the limit map of $\rho_1$. 
Now let $v_{b_1}, v_{a_1b_1}\in \mathbb{R}e_1\oplus \mathbb{R}e_2$ be eigenvectors of $\rho_1(b_1), \rho_1(a_1b_1)$ (see as matrices in $\mathsf{SL}(\mathbb{R}e_1\oplus \mathbb{R}e_2)$ via the isomorphism $e_3\mapsto e_1, e_4\mapsto e_2$) respectively, corresponding to their eigenvalues of maximum modulus (denoted by $\lambda_{b_1},\lambda_{a_1b_1}\neq 1$ respectively). The matrix $\rho_{s,t}(b_1)$ has an eigenvector with respect to its eigenvalue of maximum modulus of the form $u_{b_1}(s,t)+\overline{u}_{b_1}(s,t)$, where $u_{b_1}(s,t)\in \mathbb{R}e_1\oplus \mathbb{R}e_2$, $$u_{b_1}(s,t)=\frac{s}{\lambda_{b_1}-1}Av_{b_1}$$ and $\overline{u}_{b_1}(s,t)\in \mathbb{R}e_3\oplus \mathbb{R}e_4$. Similarly, the matrix $$\rho_{s,t}(a_1b_1)=\begin{pmatrix}
J_t & sJ_tA\\ 
 & \rho_1(a_1b_1)
\end{pmatrix}$$ has an eigenvector of the form $u_{a_1b_1}(s,t)+\overline{u}_{a_1b_1}(s,t)$, $\overline{u}_{a_1b_1}(s,t)\in \mathbb{R}e_3\oplus \mathbb{R}e_4$, $$u_{a_1b_1}(s,t)=s(\lambda_{a_1b_1}\textup{I}_2-J_t)^{-1}J_tAv_{a_1b_1}.$$ 

For every $s\neq 0$ we have that \begin{align*}\lim_{t\rightarrow 0}d_{\mathbb{P}}(u_{a_1b_1}(s,t),u_{b_1}(s,t))&=\lim_{t\rightarrow 0}d_{\mathbb{P}}\big((\lambda_{a_1b_1}\textup{I}_2-J_t)^{-1}[J_tAv_{a_1b_1}],[Av_{b_1}]\big)\\ &=d_{\mathbb{P}}([Av_{b_1}], [Av_{a_1b_1}]\big)>0,\end{align*} hence, we may choose $\varepsilon>0$ such that $\mathbb{R} u_{b_1}(s,t)\oplus \mathbb{R} u_{a_1b_1}(s,t)=\mathbb{R}e_1 \oplus \mathbb{R}e_2$, for $|t|<\varepsilon$ and $s\neq 0$. Since $\mathbb{R} e_3\oplus \mathbb{R}e_4 \subset V_{s,t}$ we have $V_{s,t}=\mathbb{R}^4$ for $s\neq 0$, $|t|<\varepsilon$ and $\xi_{s,t}^1$ is spanning.

By the definition of $d_{\mathbb{H}^2}$, there exists $R>0$, \hbox{depending only on the choice of $x_0\in \mathbb{H}_{\mathbb{R}}^2$, such that} $$\sup_{\gamma \in \pi_1(\Sigma_g)}\big||\rho_1(\gamma)|_{\mathbb{H}_{\mathbb{R}}^2}-2\log \sigma_1(\rho_1(\gamma))\big|\leq R$$ where $|\rho_1(\gamma)|_{\mathbb{H}_{\mathbb{R}}^2}:=d_{\mathbb{H}^2}(\rho_1(\gamma)x_0,x_0), \gamma \in \pi_1(\Sigma_g)$.
\medskip

\noindent {\em Claim 2. If $|t|<\varepsilon$ and $s,t\neq 0$, $\alpha_{\xi_{s,t}^1}(d_\textup{v},d_{\mathbb{P}})=\frac{1}{2}$ and $\xi_{s,t}^1:(\partial_{\infty}\mathbb{H}_{\mathbb{R}}^2,d_\textup{v})\rightarrow (\mathbb{P}(\mathbb{K}^d),d_{\mathbb{P}})$ is not $\frac{1}{2}$-H\"older.}

Observe that since $\rho_{s,t}$ is $1$-Anosov with semisimplification $\rho_{s,t}^{ss}=\begin{pmatrix}[0.7] \textup{I}_2 & \\ & \rho_1\end{pmatrix}$, there is \hbox{$R_{s,t}>0$} with the property that \begin{align*} \sigma_1(\rho_{s,t}(\gamma)) \leq R_{s,t} \sigma_1(\rho_{1}(\gamma)) \leq R_{s,t} e^{\frac{R}{2}+\frac{1}{2}|\rho_1(\gamma)|_{\mathbb{H}_{\mathbb{R}}^2}} \ \ \forall \ \gamma \in \pi_1(\Sigma_g).\end{align*} Since for every $\gamma\in \pi_1(\Sigma_g)$ we have $|\gamma|_{\mathbb{H}^2,\infty}=2\log \ell_1(\rho_1(\gamma))$ and $\xi_{s,t}^1$ is spanning, Theorem \ref{exponent} implies that $\alpha_{\xi_{s,t}^1}(d_\textup{v},d_{\mathbb{P}})=\frac{1}{2}$ for $0<|t|<\varepsilon$ and $s\neq 0$. 

\par Now it remains to prove that $\xi_{s,t}^1$ cannot be $\frac{1}{2}$-H\"older. By working as in Lemma \ref{doublebound} for the spanning map $\xi_{s,t}^1$, we may find $x,y\in \partial_{\infty}\mathbb{H}_{\mathbb{R}}^2$, $x,y\neq \lim_n \rho_1(a_1^{-n})x_0$, and $\epsilon>0$ \hbox{such that for every $n\in \mathbb{N}$} \begin{align*}\label{1-Holder'} d_{\mathbb{P}}\big(\xi_{s,t}^1(a_1^n x),\xi_{s,t}^1(a_1^n y)\big)&\geq \epsilon \frac{\sigma_2(\rho_{s,t}(a_1^n))}{\sigma_1(\rho_{s,t}(a_1^n))}\geq \epsilon\frac{\sigma_2(\rho_{s,t}(a_1^n))}{R_{s,t}e^{R/2}}e^{-\frac{1}{2}|\rho_1(\gamma)|_{\mathbb{H}_{\mathbb{R}}^2}}.\end{align*}

By Lemma \ref{Gromovproduct1}, since $x,y\neq \lim_{n}\rho_1(a_1^{-n})x_0$, there exists $C>1$, depending only on the choice of $x,y\in \partial_{\infty}\mathbb{H}_{\mathbb{R}}^2$, such that for every $n\in \mathbb{N}$: $$d_\textup{v}\big(\rho_1(a_1^n)x,\rho_1(a_1^n)y\big)\leq C e^{-|\rho_1(\gamma)|_{\mathbb{H}_{\mathbb{R}}^2}}.$$ Since $\sigma_2(\rho_{s,t}(a_1^n))=\sigma_1(\psi_{2,t}(a_1^n))=\Big|\Big| \begin{pmatrix}[0.5]
1 & tn \\ 
0 & 1
\end{pmatrix}\Big|\Big|\geq |t|n$, we conclude for every $n\in \mathbb{N}$, $$ \frac{d_{\mathbb{P}}\big(\xi_{s,t}^1(a_1^n x),\xi_{s,t}^1(a_1^n y)\big)}{\sqrt{d_\textup{v}(a_1^nx,a_1^ny)}}\geq \frac{\epsilon \sigma_2(\rho_{s,t}(a_1^n))}{\sqrt{C} R_{s,t} e^{R/2}}\geq \frac{\epsilon |t|n}{\sqrt{C} R_{s,t} e^{R/2}}$$ hence $$\sup_{z\neq w}\frac{d_{\mathbb{P}}\big(\xi_{s,t}^1(z),\xi_{s,t}^1(w)\big)}{\sqrt{d_\textup{v}(z,w)}}=+\infty.$$ Therefore, $\xi_{s,t}^1$ cannot be $\frac{1}{2}$-H\"older. The proof of the claim is complete.
\end{proof}

\begin{rmks}\normalfont{ \noindent \textup{(i)} After the construction of the examples of Theorem \ref{1-Holder}, Fran\c{c}ois Gu\'eritaud mentioned to me that he is aware of an example of a surface group $1$-Anosov representation into $\mathsf{SL}_4(\mathbb{R})$ whose Anosov limit map in $\mathbb{P}(\mathbb{R}^4)$ has H\"older exponent $1$ but fails to be $1$-H\"older.

\medskip
\noindent \textup{(ii)} Note that the previous family of representations $\{\rho_{s,t}\}_{s,t}, (s,t)\in \mathbb{R}\times (-\varepsilon, \varepsilon)$, demonstrates the failure of the continuity of the map $\mathcal{H}:\textup{Anosov}_{1,d}(\pi_1(\Sigma_g))/\mathsf{GL}_d(\mathbb{K}) \rightarrow (0,\infty)$, $$\mathcal{H}([\rho])=\alpha_{\xi_{\rho}^1}(d_\textup{v},d_{\mathbb{P}})$$ among conjugacy classes of $1$-Anosov representations of $\pi_1(\Sigma_g)$ in $\mathsf{GL}_d(\mathbb{K})$. While $\xi_{s,t}^1$ is $\frac{1}{2}$-H\"older for $s\neq 0$, the Anosov limit map of $\lim_{s,t\rightarrow 0}\rho_{s,t}=\begin{pmatrix}[0.7] \textup{I}_2 & \\ & \rho_1\end{pmatrix}$ coincides with the limit map of $\rho_1$ whose H\"older exponent is equal to $1$.}

\end{rmks}

\bibliographystyle{siam}

\bibliography{biblio.bib}

\end{document}